\documentclass[amsmath,secnumarabic,floatfix,amssymb,nofootinbib,nobibnotes,letterpaper,11pt,tightenlines]{revtex4}

\usepackage{times}

\usepackage{geometry}
\usepackage{amssymb}
\usepackage{latexsym, amsmath, amscd,amsthm}
\usepackage{graphicx}
\usepackage[percent]{overpic}
\usepackage{units}
\usepackage{hyperref}
\usepackage{diagrams}
\PassOptionsToPackage{caption=false}{subfig}
\usepackage[lofdepth]{subfig}
\usepackage{booktabs}

\usepackage{clrscode}

\newtheorem{theorem}{Theorem}

\newtheorem{lemma}[theorem]{Lemma}
\newtheorem{proposition}[theorem]{Proposition}
\newtheorem{corollary}[theorem]{Corollary}

\theoremstyle{definition}
\newtheorem{definition}[theorem]{Definition}

\newtheorem{remark}[theorem]{Remark}

\newcommand{\R}{\mathbb{R}}
\newcommand{\C}{\mathbb{C}}
\newcommand{\Q}{\mathbb{H}}
\newcommand{\Z}{\mathbb{Z}}

\newcommand{\Hopf}{\operatorname{Hopf}}
\newcommand{\FHopf}{\operatorname{FrameHopf}}

\newcommand{\Pol}{\operatorname{Pol}}
\newcommand{\hPol}{\overline{\Pol}}

\newcommand{\Arm}{\operatorname{Arm}}
\newcommand{\hArm}{\overline{\Arm}}

\newcommand{\Length}{\operatorname{Length}}

\newcommand{\cross}{\times}
\newcommand{\Sp}{\operatorname{Sp}}
\newcommand{\chord}{\operatorname{Chord}}
\newcommand{\schord}{\operatorname{sChord}}

\newcommand{\dVol}{\thinspace\operatorname{dVol}}
\newcommand{\Vol}{\operatorname{Vol}}

\newcommand{\dr}{\,\mathrm{d}r}
\newcommand{\sgyradius}{\operatorname{sGyradius}}
\newcommand{\gyradius}{\operatorname{Gyradius}}
\newcommand{\ePol}{\operatorname{ePol}}
\newcommand{\eArm}{\operatorname{eArm}}
\newcommand{\I}{\mathbf{i}}
\newcommand{\J}{\mathbf{j}}
\newcommand{\K}{\mathbf{k}}
\newcommand{\Beta}{\operatorname{\mathrm{B}}}
\newcommand{\Cov}{\operatorname{Cov}}

\def\co{\colon\thinspace}

\let\mgp=\marginpar \marginparwidth18mm \marginparsep1mm
\def\marginpar#1{\mgp{\raggedright\tiny #1}}

\let\lbl=\label
\def\label#1{\lbl{#1}\ifinner\else\marginpar{\ref{#1} #1}\ignorespaces\fi}

\begin{document}
\title[]{Probability Theory of Random Polygons from the Quaternionic Viewpoint}
\author{Jason Cantarella}
\altaffiliation{University of Georgia, Mathematics Department, Athens GA}
\noaffiliation
\author{Tetsuo Deguchi}
\altaffiliation{Ochanomizu University, Physics Department, Tokyo, Japan}
\noaffiliation
\author{Clayton Shonkwiler}
\altaffiliation{University of Georgia, Mathematics Department, Athens GA}
\noaffiliation

\begin{abstract} 
We build a new probability measure on closed space and plane polygons. The key construction is a map, given by Knutson and Hausmann using the Hopf map on quaternions, from the complex Stiefel manifold of 2-frames in $n$-space to the space of closed $n$-gons in 3-space of total length 2. Our probability measure on polygon space is defined by pushing forward Haar measure on the Stiefel manifold by this map. A similar construction yields a probability measure on plane polygons which comes from a real Stiefel manifold. 

The edgelengths of polygons sampled according to our measures obey beta distributions. This makes our polygon measures different from those usually studied, which have Gaussian or fixed edgelengths. One advantage of our measures is that we can explicitly compute expectations and moments for chordlengths and radii of gyration. Another is that direct sampling according to our measures is fast (linear in the number of edges) and easy to code. 

Some of our methods will be of independent interest in studying other probability measures on polygon spaces. We define an edge set ensemble (ESE) to be the set of polygons created by rearranging a given set of n edges. A key theorem gives a formula for the average over an ESE of the squared lengths of chords skipping $k$ vertices in terms of $k$, $n$, and the edgelengths of the ensemble. This allows one to easily compute expected values of squared chordlengths and radii of gyration for any probability measure on polygon space invariant under rearrangements of edges.  
\end{abstract}
\date{\today}
\maketitle

\section{Introduction}
In 1997, Jean-Claude Hausmann and Allen Knutson~\cite{Knutson:2_iyExxE} introduced a useful description of the space of closed $n$-edge space polygons of total length 2. They constructed a smooth surjection based on the Hopf map from the Stiefel manifold $V_2(\C^n)$ of Hermitian orthonormal 2-frames in complex $n$-space to the space of edge sets of closed $n$-edge polygons in $\R^3$ of total length 2. They proved that over ``proper" polygons with no length-zero edges, this map is locally a smooth $U(1)^n$ bundle\footnote{Recently, Howard, Manon and Millson~\cite{Howard:2008uy} explained the fiber of this map as the set of \emph{framings} of the polygon. We do not need that structure in this paper, but will use this description of framed polygons in a future paper.}. 

While Knutson and Hausmann were interested in this map primarily as a way to analyze the symplectic and algebraic geometry of polygon space, our focus is on the theory of random polygons. The main idea of this paper is to use versions of their map to push forward natural and highly symmetric probability measures from four Riemannian manifolds to four spaces of polygons:
the spheres in quaternionic and complex $n$-space map to open $n$-gons of fixed (total) length 2 in space and in the plane and the Stiefel manifolds of 2-frames in complex and real $n$-space map to closed $n$-gons of fixed (total) length 2 in space and in the plane. We call the spheres and Stiefel manifolds the ``model spaces'' for their spaces of polygons. This construction suggests a natural measure on polygon spaces that does not seem to have been studied before: the measure pushed forward from Haar measure on the model spaces. Since Haar measure is maximally symmetric, these measures are mathematically fundamental and may be physically significant. In this context, we find it promising that in the case of an equilateral $n$-edge polygon, our measure restricts to the standard probability measure on arm space: the product measure on $n$ copies of the standard $S^2$. Further, for a closed equilateral polygon, our measure again restricts to the expected one: it is the subspace measure on $n$-tuples of vectors in the round $S^2$ which sum to zero. We note that alternate measures on the polygon spaces can be constructed in the same way by choosing different measures on the model spaces (such as the various measures on Stiefel manifolds found in \cite{Chikuse:2003we}).  

The most important practical property of these measures is that it is very easy to directly sample $n$-edge closed polygons in $O(n)$ time (the constant is small), allowing us to experiment with very large and high-quality ensembles of polygons. The most important theoretical property of this measure is that it is highly symmetric, allowing us to prove theorems which match our experiments. We will be able to define a transitive measure-preserving action of the full unitary group $U(n)$ on $n$-edge closed space polygons of length 2. Using these symmetries, we will be able to explicitly compute simple exact formulae for the expected values of squared chord lengths and radii of gyration for random open and closed polygons of fixed length, with corresponding formulae for equilateral polygons. We can then obtain explicit bounds on how fast the chord lengths of a closed polygon converge to those of an open polygon as the number of edges increases, providing rigorous justification for the intuition that a sufficiently long polygon ``forgets'' that it is closed. 

The advantages of the present method for generating random polygons should be quite important in the study of ring polymers in polymer physics. While this is not a standard model for random polygons, our scaling results agree with the results for equilateral polygons in~\cite{Kramers:1946ki}, \cite{1949JChPh..17.1301Z}, and \cite{Casassa:1965tc}. Orlandini and Whittington~\cite{Orlandini:2007kn} give an excellent survey on what is known about the effects of topological constraints (such as knotting or linking) on the behavior of ring polymers. While we do not analyze knotting and linking here, our hope is that by providing an exactly solvable model together with an efficient algorithm for unbiased sampling, we can eventually answer some of the open theoretical questions in this field. Furthermore, ring polymers of large molecular weights with small dispersion have been synthesized quite recently, by which one can experimentally confirm theoretical predictions. 

\clearpage

\section{Polygonal Arm Spaces, Moduli Spaces, and Quaternions}
\label{sec:armSpace}

We are interested in a number of spaces of polygons in this paper. For convenience, all of our polygon spaces will be composed of polygons with total length 2 (though our results apply, by scaling, to polygon spaces of any fixed length).
\begin{definition}
Let $\Arm_3(n)$ be the moduli space of $n$-edge polygons (which may not be closed) of length 2 up to translation in $\R^3$, and $\Arm_2(n)$ be the corresponding space of planar polygons. We note that by fixing a plane in $\R^3$, we have $\Arm_2(n) \subset \Arm_3(n)$. If we identify polygons related by a rotation, we have the commutative diagram:
\begin{equation}
\begin{diagram}
SO(3)   & \rTo   & \Arm_3(n)  & \rOnto  & \Arm_3(n)/SO(3) = \hArm_3(n)  \\
\uTo  &      & \uTo     &         & \uTo                      \\
O(2)    & \rTo   & \Arm_2(n)  & \rOnto  & \Arm_2(n)/O(2) = \hArm_2(n)  \\
\end{diagram}
\label{eq:armTwoThree}
\end{equation}
where $\hArm_3(n)$ is the moduli space of $n$-edge polygons of length 2 up to translation and rotation in $\R^3$, but $\hArm_2(n)$ is the space of $n$-edge polygons of length 2 up to translation, rotation, \emph{and reflection} in $\R^2$. This additional identification is needed to make $\hArm_2(n) \subset \hArm_3(n)$, since two polygons related by a reflection in the plane are related by a rotation in space. \end{definition}
An element of $\Arm_i(n)$ is a list of edge vectors $e_1, \dots, e_n$ in $\R^i$ whose lengths sum to $2$, while an element of $\hArm_i(n)$ is an equivalence class of edge lists. 

We now need some special properties of $\R^2$ and $\R^3$. Recall that the skew-algebra of quaternions is defined by adding formal elements $\I$, $\J$, and $\K$ to $\R$ with the relations that $\I^2 = \J^2 = \K^2 = -1$ and $\I\J = \K$, $\J\K = \I$, and $\K\I = \J$ while $\J\I = -\K$, $\K\J = -\I$, and $\I\K = -\J$. Using these rules, quaternionic multiplication is an associative (but not commutative) multiplication on $\R^4$ which makes $\R^4$ into a division algebra. As with complex numbers, we refer to the real and imaginary parts of a quaternion, though here the imaginary part is a 3-vector determined by the three coefficients of $\I$, $\J$, and $\K$. If we identify $\R^3$ with the imaginary quaternions, the unit quaternions ($S^3$) double-cover the orthonormal 3-frames ($SO(3)$) via the triple of Hopf maps $\Hopf_\I(q) = \bar{q}{\I}q$, $\Hopf_\J(q) = \bar{q}{\J}q$, and $\Hopf_\K(q) = \bar{q}{\K}q$. Since we will focus on the map $\Hopf_\I$ as our ``standard'' Hopf map, we denote $\Hopf_\I$ by $\Hopf$. 
\begin{multline}
\label{eq:hopf1}
q = (q_0,q_1,q_2,q_3) \mapsto \\
\left( 
\begin{array}{ccc}
\bar{q}{\I}q & \bar{q}{\J}q & \bar{q}{\K}q \\
\end{array}
\right)
=
\left(
\begin{array}{ccc}
 q_0^2+q_1^2-q_2^2-q_3^2 & 2 q_1 q_2+2 q_0
   q_3 & 2 q_1 q_3-2 q_0 q_2 \\
 2 q_1 q_2-2 q_0 q_3 &
   q_0^2-q_1^2+q_2^2-q_3^2 & 2 q_0 q_1+2 q_2
   q_3 \\
 2 q_0 q_2+2 q_1 q_3 & 2 q_2 q_3-2 q_0
   q_1 & q_0^2-q_1^2-q_2^2+q_3^2
\end{array}
\right)
\end{multline}

It is a standard computation to verify that $\bar{q}{\I}q$ (and the other columns) are purely imaginary quaternions with $\I$, $\J$ and $\K$ components matching the real matrix form at right, that they are all orthogonal in~$\R^3$, and that the norm of each column is the square of the norm of $q$. Further, if we call this map the ``frame Hopf map'' $\FHopf(q)$, then it is also standard that this is a covering map and  
\begin{equation}
\FHopf(q) = \FHopf(q') \iff q = \pm q'
\end{equation}

If we write a unit quaternion in the form $q = (\cos \nicefrac{\theta}{2}, \sin \nicefrac{\theta}{2}\, \vec{n})$, then the image ${\FHopf(q) \subset SO(3)}$ is the rotation around the axis $\vec{n}$ by angle $\theta$. The action of this rotation matrix on $\R^3$ corresponds to quaternionic conjugation by $q$. That is, explicitly, if we view a purely imaginary quaternion $p \in I\Q$ as a vector $\vec{p} \in \R^3$, we have
\begin{equation}
\FHopf(q) \vec{p} = \bar{q} p q.
\label{eq:conjugation}
\end{equation}

There is one last way of writing unit quaternions which will be important to us: the unit quaternions can be identified with the special unitary group $SU(2)$ by representing quaternions as \emph{Pauli matrices}. If we let $q = a + b \J$, where $a, b \in \C$, then the map $\eta \co \Q \rightarrow \operatorname{Mat}_{2 \cross 2}(\C)$ given by
\begin{equation}
\eta(q) = 
\left(
\begin{matrix}
a & b \\
-\bar{b} & \bar{a}
\end{matrix}
\right),
\,\,
\eta(1) = 
\left(
\begin{matrix}
1 & 0 \\
0 & 1
\end{matrix}
\right),
\,\, 
\eta(\I) = 
\left(
\begin{matrix}
i & 0 \\
0 & -i
\end{matrix}
\right),
\,\,
\eta(\J) = 
\left(
\begin{matrix}
0 & 1 \\
-1 & 0
\end{matrix}
\right),
\,\,
\eta(\K) = 
\left(
\begin{matrix}
0 & i \\
i & 0
\end{matrix}
\right)
\label{eq:pauliform}
\end{equation}
is an injective $\R$-algebra homomorphism. In particular, $\eta(\bar{q}) = \eta(q)^*$, $\eta(pq) = \eta(p) \eta(q)$ and $\det(\eta(q)) = q\bar{q} = |q|^2$, so $\eta$ is an isomorphism between the groups $U\Q$ and $SU(2)$. 

We can extend the Hopf map coordinatewise to a map $\Hopf \co \Q^n \rightarrow \R^{3n}$. 
\begin{proposition}
$\Hopf$ is a smooth map from the sphere of radius $\sqrt{2}$, $S^{4n-1} \subset \Q^n$ onto $\Arm_3(n)$. 
\end{proposition}

\begin{proof} Given $\vec{q} \in \Q^n$, we define the edge set of the polygon $\Hopf(\vec{q})$ by 
	\[
		(e_1, \dots, e_n) := (\Hopf(q_1), \dots, \Hopf(q_n)).
	\]
The only thing to check is that
\begin{equation}
\Length(P) = \sum |e_i| = \sum |q_i|^2 = 2.
\end{equation}
This follows from the fact that the Hopf map squares norms.
\end{proof}

We now consider the moduli space $\hArm_3(n) = \Arm_3(n)/SO(3)$. By~\eqref{eq:conjugation}, for any unit quaternion $p$ we have
\begin{equation*}
\Hopf(qp) = \bar{p} \Hopf(q) p = \FHopf(p) \Hopf(q).
\end{equation*}
This means that the $\Hopf$ map takes equivalence classes of points in $\Q^n$ under right-multiplication by a unit quaternion to equivalence classes of edge vectors under the action of $SO(3)$. That is, $\Hopf$ maps points in the quaternionic projective space $\Q P^{n-1} = S^{4n-1}/Sp(1)$ to $\hArm_3(n)$. Here we use $Sp(1)$ to refer to the group of unit quaternions. This yields the commutative diagram
\begin{equation}
\begin{diagram}[h=0.35in]
Sp(1) = U\Q   & \rTo   & S^{4n-1} \subset \Q^n & \rOnto  & \Q P^{n-1} = S^{4n-1}/Sp(1) \\
\dOnto^{\FHopf}  &        & \dOnto^{\Hopf}  &       & \dOnto                       \\
SO(3)   & \rTo   & \Arm_3(n)  & \rOnto  & \Arm_3(n)/SO(3) = \hArm_3(n)   \\
\end{diagram}
\label{eq:armThree}
\end{equation}

Our spaces of planar arms fit naturally into this framework. Consider the planes $1 \oplus \J$ and $\I (1 \oplus \J)  = \I \oplus \K \subset \Q$. The Hopf map sends each of these to the $\I \oplus \K$ plane:
\begin{align}
\Hopf(a + b\J) &= (a - b\J) \I (a + b\J) = 
%(a - b\J) (a \I + b \K) = 
(a^2 - b^2) \I + 2 ab \K = \I \overline{(a + b \J)^2} \\
\Hopf(a\I + b\K) &= (-a\I - b\K) \I (a\I + b\K) = 
%-(a\I + b\K) (-a - b\J) = 
(a^2 - b^2) \I - 2 ab \K = \I (a + b \J)^2 
\label{eq:hopfcoords}
\end{align}
That is, if we think of $z = a + b \J$, then $\Hopf(z) = \I \overline{z}^2$ and $\Hopf(\I z) = \I z^2$. 

The planes $1 \oplus \J$ and $\I \oplus \K$ are preserved by quaternionic multiplication by $\I$ (which exchanges them) and by multiplication by any unit quaternion of the form $a + b\J$, which rotates each plane. If we think of unit quaternions as elements of $SU(2)$ using the $\eta$ map of~\eqref{eq:pauliform}, these matrices form a copy of $O(2)$ given by the subspace $SO(2) \subset SU(2)$ of matrices with real entries \emph{and} the complex matrices $\eta(\pm \I)$ which exchange purely real and purely imaginary matrices while negating one of the columns. 

We now have a diagram corresponding to~\eqref{eq:armThree} for planar polygons. The edge lists of polygons in the $\I \oplus \K$ plane are all Hopf images of points in the disjoint union of the complex spheres of radius $\sqrt{2}$ in the $2n$-dimensional subspaces of $\Q^n$ given by $(1 \oplus \J)^n$ and $(\I \oplus \K)^n$. These spheres are exchanged by the action of $\eta(\pm\I)$ and their (complex) coordinates are rotated by the action of the real matrices $SO(2) \subset SU(2)$. 
\begin{equation}
\begin{diagram}[h=0.35in]
SO(2) \cross \eta(\pm\I) = O(2)  & \rTo   & S^{2n-1} \sqcup S^{2n-1} \subset \Q^n & \rOnto  & S^{2n-1} \sqcup S^{2n-1}/O(2) \\
\dOnto^{\FHopf}  &        & \dOnto^{\Hopf}  &       & \dOnto                       \\
O(2)   & \rTo   & \Arm_2(n)  & \rOnto  & \Arm_2(n)/O(2) = \hArm_2(n)  \\
\end{diagram}
\raisetag{10pt}
\label{eq:armTwo}
\end{equation}
We note that while $\Hopf \co S^{2n-1} \sqcup S^{2n-1} \rightarrow \Arm_2(n)$ is surjective, these spheres are not all of the inverse image of $\Arm_2(n)$ in the unit sphere $S^{4n-1} \subset \Q^n$. There are also $2^{n-1}$ ``mixed'' spheres where some quaternionic coordinates lie in the $1 \oplus \J$ plane and others lie in the $\I \oplus \K$ plane who project to planar polygons. We won't need these extra spheres to define our measure on $\Arm_2(n)$.

Putting this all together, we have
\begin{proposition}
The diagrams~\eqref{eq:armThree} and~\eqref{eq:armTwo} can be joined with~\eqref{eq:armTwoThree} to form a large commutative diagram:
\newarrow{Is}{<}{-}{-}{-}{>}
\begin{equation*}
\begin{diagram}[w=0.45in,h=0.35in,tight]
SU(2)       & \rTo   &             &      & S^{4n-1} \subset \Q^n & \rTo   &           &        & S^{4n-1}/\Sp(1) = \Q P^{n-1}    &        & \\
             & \luTo  &     &      & \vLine                & \luTo  &           &        & \vLine             & \luTo  & \\
\dOnto   &        & O(2)       &      & \HonV                 & \rTo   & S^{2n-1} \sqcup S^{2n-1} & \rTo   & \HonV              &        & S^{2n-1} \sqcup S^{2n-1}/O(2) \\
             &        & \dOnto^{\FHopf}  &      & \dOnto_{\Hopf}          &        &           &        & \dOnto               &        &   \\
SO(3)        & \hLine & \VonH       & \rTo & \Arm_3(n)             & \hLine & \VonH     & \rOnto   & \hArm_3(n)         &        & \dOnto \\
             & \luTo  &     &      &                       & \luTo& \dOnto>{\Hopf}&      &                    & \luTo  & \\
             &        & O(2)        &      & \rTo                  &        & \Arm_2(n) &        & \rOnto               &        & \hArm_2(n) \\
\end{diagram}
\end{equation*}
\end{proposition}

\begin{proof} 
The only thing left to check is that the top right arrow $S^{2n-1} \sqcup S^{2n-1}/O(2) \hookrightarrow \Q P^{n-1}$ is well defined. This follows from the fact that our embedding of $O(2)$ in $SU(2)$ is the stabilizer of this subset of $S^{4n-1} \subset \Q^n$ in the group action of $SU(2) = \Sp(1)$ on $\Q^n$.
\end{proof}

\section{Closed Polygon Spaces and Stiefel Manifolds}

Now that we understand arm spaces from the quaternionic point of view, we turn to closed polygon spaces as subspaces of the arm spaces. It is easiest to see closed polygons in context by defining $\Arm_i(n,\ell)$ to be the subspace of $\Arm_i(n)$ of polygons which fail to close by length $\ell$ and then letting $\Pol_i(n) = \Arm_i(n,0)$. As before $\hPol_3(n) = \Pol_3(n)/SO(3)$ while $\hPol_2(n) = \Pol_2(n)/O(2)$, since we want $\hPol_2(n) \subset \hPol_3(n)$.

We now describe the fiber $\Hopf^{-1}(\Pol_3(n))$ as a subspace of the quaternionic sphere $S^{4n-1}$. To do so, we write the quaternionic $n$-sphere $S^{4n-1}$ as the join $S^{2n-1} \star S^{2n-1}$ of complex $n$-spheres. The join map is given in coordinates by 
\begin{equation}
(u,v,\theta) \mapsto \sqrt{2} (\cos \theta u + \sin \theta v \J)
\end{equation}
where $u$, $v \in \C^{n}$ lie in the unit sphere and $\theta \in [0,\pi/2]$. Now consider the subspace of the quaternionic sphere described by $\{(u,v,\pi/4) \,|\, \left<u,v\right> = 0\}$. This subspace is naturally identified with the Stiefel manifold $V_2(\C^n)$ of Hermitian orthonormal 2-frames $(u,v)$ in $\C^n$, and in fact the subspace metric on $V_2(\C^n)$ agrees with its standard Riemannian metric. It is the inverse image of $\Pol_3(n)$ under $\Hopf$, as shown by Knutson and Hausmann.

\clearpage

\begin{proposition}[\cite{Knutson:2_iyExxE}]
The coordinatewise Hopf map $\Hopf$ takes the Stiefel manifold ${V_2(C^n) \subset \C^n \cross \C^n = \Q^n}$ onto $\Pol_3(n)$. The $SU(2)$ or $Sp(1)$ action on $\Q^n$ preserves $V_2(\C^n)$ (and is the standard action of $SU(2)$ on $V_2(\C^n)$, rotating the two basis vectors in their common plane) and descends to the $SO(3)$ action on $\Pol_3(n)$, leading to a commutative diagram:
\begin{equation}
\begin{diagram}[h=0.35in]
SU(2)   & \rTo   & V_2(\C^n) \subset \Q^n & \rOnto  & V_2(\C^n)/SU(2) \\
\dOnto^{\FHopf}  &        & \dOnto^{\Hopf}  &       & \dOnto                       \\
SO(3)   & \rTo   & \Pol_3(n)  & \rOnto  & \Pol_3(n)/SO(3) = \hPol_3(n)   \\
\end{diagram}
\label{eq:polThree}
\end{equation}
\label{prop:ell}
\end{proposition}

\begin{proof}
In complex form, the map $\Hopf_\I(q)$ can be written more simply as 
\begin{equation}
\label{eq:coords}
\begin{split}
\Hopf_\I(a + b \J) &= (a\bar{a} - b\bar{b},-i(a\bar{b} - \bar{a}b), a\bar{b} + \bar{a}b) \\
				  &= (|a|^2 - |b|^2, 2\Im(a\bar{b}),2\Re(a\bar{b})) \\
				  &= \I (|a|^2 - |b|^2 + 2 a \bar{b} \J )
\end{split}
\end{equation}
using the identification of $\R^3$ with the imaginary quaternions. This means that the vector connecting the first and last vertices of $P = \Hopf(q_1,\dots,q_n)$ has norm 
\begin{equation}\label{eq:arm}
	\begin{split}
		\left|\sum e_i\right|^2 = \left|\sum \Hopf(q_i)\right|^2 &= \left| \sum 2 |\cos \theta \, u_i|^2 - \sum 2 |\sin \theta \, v_i|^2 + 4 \cos \theta \sin \theta \, \sum \overline{u_i} v_i \J \right|^2 \\
		&=  \left| 2\cos^2 \theta - 2 \sin^2 \theta\right|^2 + \left|4 \cos \theta \sin \theta \left< u, v\right> \right|^2 \\
		&=  \left| 2 \cos 2\theta\right|^2 + \left|2 \sin 2\theta \left< u,v \right> \right|^2 \\
		&=  4 \cos^2 2\theta + 4 \sin^2 2\theta \left| \left< u,v \right> \right|^2.
	\end{split}
	\raisetag{36pt}
\end{equation}
Thus the polygon closes if and only if $\theta = \pi/4$ and $u$, $v$ are orthogonal. 
\end{proof}

We note in passing that $V_2(\C^n)/SU(2)$ is not quite the complex Grassmann manifold $G_2(\C^n)$ of complex 2-planes in $\C^n$. In fact, $V_2(\C^n)/SU(2)$ is a circle bundle over ${G_2(\C^n) = V_2(\C^n)/U(2)}$. Howard, Manon, and Millson~\cite{Howard:2008uy} identified $G_2(\C^n)$ as essentially a covering space of the quotient of the moduli space of \emph{framed} closed polygons in $\R^3$ by the circle action given by simultaneous rotation of all vectors in the frame. 

A corresponding theorem holds for closed planar polygons under the action of $O(2)$ rather than $SO(2)$. In the language of Howard, Manon, and Millson, these are not planar polygons framed in $3$-space, but rather planar polygons framed with respect to the plane.

\clearpage

\begin{proposition}[\cite{Knutson:2_iyExxE}]
The coordinatewise Hopf map $\Hopf$ takes the disconnected manifold ${V_2(\R^n) \sqcup \I V_2(\R^n) \subset \C^n \cross \C^n = \Q^n}$ onto $\Pol_2(n)$. The action of the orthogonal group ${O(2) = SO(2) \cross \eta(\pm\I) \subset SU(2)}$ on $\Q^n$ preserves $V_2(\R^n) \sqcup \I V_2(\R^n)$. The quotient space $(V_2(\R^n) \sqcup \I V_2(\R^n))/O(2)$ is the Grassmann manifold $G_2(\R^n)$ of 2-planes in $\R^n$ and we have a commutative diagram:
\begin{equation}
\begin{diagram}[h=0.35in]
O(2) \subset SU(2)   & \rTo   & V_2(\R^n) \sqcup \I V_2(\R^n) \subset \C^n \cross \C^n = \Q^n & \rOnto  & G_2(\R^n) \\
\dOnto_{\FHopf}  &        & \dOnto^{\Hopf}  &       & \dOnto                       \\
O(2)   & \rTo   & \Pol_2(n)  & \rOnto  & \Pol_2(n)/O(2) = \hPol_2(n)  \\
\end{diagram}
\label{eq:polTwo}
\end{equation}
\end{proposition}

\begin{proof} 
This follows by combining of our characterization of closed polygons from Proposition~\ref{prop:ell} and our characterization of planar polygons from~\eqref{eq:armTwo}. We note that as above, this is not the entire inverse image of $\Pol_2(n)$ under the Hopf map. Just as ``mixed'' spheres with some quaternionic coordinates in $1 \oplus \J$ and some in $\I \oplus \K$ map to $\Arm_2(n)$, ``mixed'' frames $(a,b) \in V_2(\C^n)$ with some pairs $(a_i,b_i)$ purely real and others purely imaginary map to $\Pol_2(n)$. We will not need these additional frames to define our measure below. 
\end{proof}

As before, the inclusion of planar polygons into space polygons can now be extended to a large commutative diagram:
\begin{proposition}
The diagrams~\eqref{eq:polThree} and~\eqref{eq:polTwo} can be joined by inclusions to form a large commutative diagram:
\newarrow{Is}{<}{-}{-}{-}{>}
\begin{equation*}
\begin{diagram}[w=0.45in,h=0.35in,tight]
SU(2)       & \rTo   &             &      & V_2(\C^n) \subset \Q^n & \rTo   &           &        & V_2(\C^n)/SU(2)    &        & \\
             & \luTo  &     &      & \vLine                & \luTo  &           &        & \vLine             & \luTo  & \\
\dOnto   &        & O(2)       &      & \HonV                 & \rTo   & V_2(\R^n) \sqcup \I V_2(\R^n) & \rTo   & \HonV              &        & G_2(\R^n) \\
             &        & \dOnto^{\FHopf}  &      & \dOnto_{\Hopf}          &        &           &        & \dOnto               &        &   \\
SO(3)        & \hLine & \VonH       & \rTo & \Pol_3(n)             & \hLine & \VonH     & \rOnto   & \hPol_3(n)         &        & \dOnto \\
             & \luTo  &     &      &                       & \luTo& \dOnto>{\Hopf}&      &                    & \luTo  & \\
             &        & O(2)        &      & \rTo                  &        & \Pol_2(n) &        & \rOnto               &        & \hPol_2(n) \\
\end{diagram}
\end{equation*}
\end{proposition}

\clearpage

Finally, we note that the inclusion of closed polygons into arm space generates yet another set of useful commutative diagrams. For space polygons, we have
\newarrow{Is}{<}{-}{-}{-}{>}
\begin{equation*}
\begin{diagram}[w=0.45in,h=0.35in,tight]
\Sp(1)       & \rTo   &             &      & S^{4n-1} \subset \Q^n & \rTo   &           &        & \Q P^{n-1}    &        & \\
             & \luIs^{\cong}  &     &      & \vLine                & \luTo  &           &        & \vLine             & \luTo  & \\
\dTo^{\pi}   &        & SU(2)       &      & \HonV                 & \rTo   & V_2(\C^n) & \rTo   & \HonV              &        & V_2(\C^n)/SU(2) \\
             &        & \dTo^{\pi}  &      & \dTo^{\Hopf}          &        &           &        & \dTo               &        &   \\
SO(3)        & \hLine & \VonH       & \rTo & \Arm_3(n)             & \hLine & \VonH     & \rTo   & \hArm_3(n)         &        & \dTo \\
             & \luIs^{\cong}  &     &      &                       & \luTo& \dTo>{\Hopf}&      &                    & \luTo  & \\
             &        & SO(3)       &      & \rTo                  &        & \Pol_3(n) &        & \rTo               &        & \hPol_3(n) \\
\end{diagram}
\end{equation*}
while there is a corresponding diagram (not shown) for planar polygons. Of course, we could combine the two into a single mighty diagram connecting all two dozen spaces at hand, but we refrain out of consideration for the reader.

\section{Symmetric Measures on Polygon Spaces}

We now see the model spaces referred to in the introduction. The Hopf map now maps four standard Riemannian manifolds to four polygon spaces:
\newarrow{Ishaft}{C}{-}{-}{-}{-}
\begin{equation*}
\begin{diagram}[width=0.45in,height=0.35in,tight]
S^{2n-1} \sqcup S^{2n-1}  & \rTo   &                               &      & S^{4n-1} \subset \Q^n &          &              \\
                          & \luTo  &                               &      & \vLine                & \luTo  &              \\
\dTo^{\Hopf}              &          & V_2(\R^n) \sqcup \I V_2(\R^n) &      & \HonV                 & \rTo   & V_2(\C^n)    \\
                          &          & \dTo^{\Hopf}                  &      & \dTo^{\Hopf}          &          &              \\
\Arm_2(n)                 & \hLine & \VonH                         &\rTo& \Arm_3(n)             &          &              \\
                          & \luTo  &                               &      &                       & \luTo  & \dTo>{\Hopf} \\
                          &          & \Pol_2(n)                     &      & \rTo                &          & \Pol_3(n)    \\
\end{diagram}
\end{equation*}
We can now define probability measures on the polygon spaces by pushing forward measures on the model spaces. While we are free to choose any measure in this construction, since each of the model spaces is a symmetric space it is natural to choose to push forward Haar measure. Of course, this is also the measure defined by the standard Riemannian metrics on these spaces.
\begin{definition}
We define the \emph{symmetric measureŒ} $\mu$ on our polygon spaces by 
\begin{align*}
\text{for $U \subset \Arm_3(n)$,} \quad \mu(U) &= \frac{1}{\Vol S^{4n-1}} \int_{\Hopf^{-1}(U)} \dVol_{S^{4n-1} \subset \Q^n}, \\ 
\text{for $U \subset \Pol_3(n)$,} \quad \mu(U) &= \frac{1}{\Vol V_2(\C^n)} \int_{\Hopf^{-1}(U)} \dVol_{V_2(\C^n)}, \\
\text{for $U \subset \Arm_2(n)$,} \quad \mu(U) &= \frac{1}{\Vol (S^{2n-1} \sqcup S^{2n-1})} \int_{\Hopf^{-1}(U)} \dVol_{S^{2n-1} \sqcup S^{2n-1} \subset \Q^n}, \\ 
\text{for $U \subset \Pol_2(n)$,} \quad \mu(U) &= \frac{1}{\Vol (V_2(\R^n) \sqcup \I V_2(\R^n))} \int_{\Hopf^{-1}(U)} \dVol_{V_2(\R^n) \sqcup \I V_2(\R^n)}. 
\end{align*}
\label{def:symmeasure}
\end{definition}
\vspace{-0.1in}
Since each of these manifolds has a transitive group of isometries, it will prove relatively easy to integrate over these spaces. For instance, our measure on $\Arm_3(n)$ is preserved by the action of the quaternionic unitary group $Sp(n)$ on $S^{4n-1}$, while our measure on $\Pol_3(n)$ is preserved by the action of $U(n)$ on $V_2(\C^n)$. Our measure on $\Arm_2(n)$ is preserved by the action of $U(n)$ on each complex sphere, as well as by the $\Z/2\Z$ action exchanging them, while our measure on $\Pol_2(n)$ is preserved by the $O(n)$ action on each $V_2(\R^n)$, as well as by the $\Z/2\Z$ action exchanging them.

We note that these measures push forward to corresponding measures on the smaller spaces $\hPol_i(n)$ and $\hArm_i(n)$. This fact turns out to be relatively unimportant for computing expectations, since it seems easier to integrate over the larger spaces. The real importance of this construction is likely to be theoretical: any function on plane polygons which is invariant under the full Euclidean group $O(2) \cross \R^2$ now lifts to a function on $G_2(\R^n)$, while any function on space polygons which is invariant under orientation preserving isometries $SO(3) \cross \R^3$ now lifts to a function on $V_2(\C^n)/SU(2)$. This seems potentially fascinating! For instance, what are the properties of the writhing number as a map $\operatorname{Writhe} \co V_2(\C^n)/SU(2) \rightarrow \R$?

\section{Moments of the Edgelength Distribution on Arm and Polygon Space}
\label{sec:moments}

Since the Hopf map squares norms, given a point $(q_1, \dots, q_n)$ in the quaternionic $n$-sphere $S^{4n-1}$, the edges of the corresponding polygon have lengths $|q_1|^2, \dots, |q_n|^2$. We now compute the moments of the edgelength distribution on arm space and on polygon space using a formula of Lord~\cite{Lord:1954wh} relating the moments $\mu_2$, $\mu_4$, $\dots$ of a spherical distribution on $k$-space to the moments $\nu_2$, $\nu_4$, \dots of its projection onto an $l$-dimensional subspace:
\begin{equation*}
\frac{\mu_2}{k} = \frac{\nu_2}{l}, \quad \frac{\mu_4}{k(k+2)} = \frac{\nu_4}{l(l+2)}, \quad \frac{\mu_6}{k(k+2)(k+4)} = \frac{\nu_6}{l(l+2)(l+4)}, \dots
\end{equation*}

This can be packaged into the following general formula either by doing some arithmetic on the above or by integrating~\cite[Equation (8)]{Lord:1954p9280}:
\[
	\frac{1}{2}\mu_{2p}\Beta(p,\nicefrac{k}{2}) = \frac{1}{2}\nu_{2p}\Beta(p,\nicefrac{l}{2}),
\]
where $\Beta$ is the Euler beta function.

We can now easily compute moments of the edgelength distribution on our arm and polygon spaces. We will later give an explicit probability density function for these distributions in Proposition~\ref{prop:pdfs}.

\begin{proposition}\label{prop:moments}
	The moments of the distribution of an edgelength $|e_i|$ are
	\begin{align*}
		E(|e_i|^p,\Arm_3(n)) & = 2^p\frac{\Beta(p,2n)}{\Beta(p,2)},\quad  E(|e_i|^p,\Pol_3(n)) = \frac{\Beta(p,n)}{\Beta(p,2)} \\
		E(|e_i|^p,\Arm_2(n)) & = 2^p \frac{\Beta(p,n)}{\Beta(p,1)} , \quad \ \  E(|e_i|^p,\Pol_2(n)) = \frac{\Beta(p,\nicefrac{n}{2})}{\Beta(p,1)}.
	\end{align*}
\end{proposition}

Using Stirling's formula, we get the following approximations for large $n$:
\begin{align*}
	E(|e_i|^p,\Arm_3(n)) & \simeq \frac{(p+1)!}{n^p} \simeq E(|e_i|^p,\Pol_3(n)) \\
	E(|e_i|^p,\Arm_2(n)) & \simeq \frac{2^p p!}{n^p}  \simeq E(|e_i|^p,\Pol_2(n))
\end{align*}

% \begin{proposition}
% The first moments of the distribution of an edgelength $|e_i|$ are 
% \begin{align*}
% E(|e_i|,\Arm_3(n)) &= \frac{2}{n}, \quad E(|e_i|,\Pol_3(n)) = \frac{2}{n}, \\
% E(|e_i|,\Arm_2(n)) &= \frac{2}{n}, \quad E(|e_i|,\Pol_2(n)) = \frac{2}{n}. 
% \end{align*}
% The second moments of $|e_i|$ are
% \begin{align*}
% E(|e_i|^2,\Arm_3(n)) &= \frac{6}{n(n+\nicefrac{1}{2})}, \quad E(|e_i|^2,\Pol_3(n)) = \frac{6}{n(n+1)}, \\
% E(|e_i|^2,\Arm_2(n)) &= \frac{8}{n(n+1)}, \quad E(|e_i|^2,\Pol_2(n)) = \frac{8}{n(n+2)}. 
% \end{align*}
% The third moments of $|e_i|$ are
% \begin{align*}
% E(|e_i|^3,\Arm_3(n)) &= \frac{24}{n(n+\nicefrac{1}{2})(n+1)}, \quad E(|e_i|^3,\Pol_3(n)) = \frac{24}{n(n+1)(n+2)}, \\
% E(|e_i|^3,\Arm_2(n)) &= \frac{48}{n(n+1)(n+2)}, \quad E(|e_i|^3,\Pol_2(n)) = \frac{48}{n(n+2)(n+4)}.  
% \end{align*}
% Further moments can be easily computed as below.
% \label{prop:moments}
% \end{proposition}

\begin{proof}[Proof of Proposition~\ref{prop:moments}]
	The $p$th moment of edgelength for space arms is the $2p$th moment $\nu_{2p}$ of the distribution on $\Q^1 = \R^4$ obtained by projecting the uniform measure on the quaternionic $(4n-1)$-sphere of radius $\sqrt{2}$ onto quaternionic $1$-space. Since the measure on the sphere has $2p$th moment $\sqrt{2}^{2p} = 2^p$, according to Lord's formula we have
	\begin{equation*}
		\frac{1}{2} 2^p \Beta(p,2n) = \frac{1}{2} \nu_{2p} \Beta(p,2),
	\end{equation*}
	so
	\[
		\nu_{2p} = 2^p \frac{\Beta(p,2n)}{\Beta(p,2)}.
	\]
	
	For the $\Arm_2(n)$ spaces, we are projecting from the $\sqrt{2}$-sphere in $\C^n$ to $\C^1$, so the calculations become
	\[
		\frac{1}{2}2^p\Beta(p,n) = \frac{1}{2}\nu_{2p}\Beta(p,1) \implies \nu_{2p} = 2^p \frac{\Beta(p,n)}{\Beta(p,1)}.
	\]
	
	Repeating these calculations on the Stiefel manifolds representing $\Pol_3(n)$ and $\Pol_2(n)$ requires only a little more work. Given a pair $(a,b) \in V_2(\C^n)$ representing a polygon in $\Pol_3(n)$, the edgelength is given by $|e_i| = |a_i|^2 + |b_i|^2$, or the squared norm of the vector $(a_i,b_i) \in \C^2$. If $(a,b)$ is uniformly distributed on the Stiefel manifold, the vector $a$ is uniformly distributed on the unit $S^{2n-1}$, so the projection from $(a,b) \mapsto a_i \in \C$ has $2p$th moment $\nu_{2p}$ obeying
	\[
		\frac{1}{2} \cdot 1 \cdot \Beta(p,n) = \frac{1}{2}\nu_{2p} \Beta(p,1) \implies \nu_{2p} = \frac{\Beta(p,n)}{\Beta(p,1)}.
	\]
	
	On the other hand, the measure on $V_2(\C^n)$ is $U(2)$ invariant, so the projection $(a,b) \mapsto (a_i,b_i) \in \C^2$ projects the uniform measure on $V_2(\C^n)$ to a spherically symmetric measure on $\C^2 = \R^4$. The projection from $\C^2$ to $\C^1$ given by $(a_i,b_i) \mapsto a_i$ takes this unknown measure to the measure on $\C$ whose moments we computed above. Thus, we can apply Lord's formula again to solve backwards for the moments $\mu_{2p}$ of edgelength:
	\[
		\frac{1}{2}\mu_{2p} \Beta(p,2) = \frac{1}{2} \left(\frac{\Beta(p,n)}{\Beta(p,1)}\right)\Beta(p,1) \implies \mu_{2p} = \frac{\Beta(p,n)}{\Beta(p,2)}.
	\]
	
	For planar polygons, the calculations are similar, but the moments of the projected measure on $a_i \in \R$ are 
	\begin{equation*}
	\nu_{2p} = \frac{\Beta(p,\nicefrac{n}{2})}{\Beta(p,\nicefrac{1}{2})},
	\end{equation*}
	and the same ``backwards'' application of Lord's formula works as above to solve for the moments of the unknown distribution on $\C = \R^2$ from these moments of the distribution on $\R$.
\end{proof}

Notice that in all cases the first moment of edgelength is $\nicefrac{2}{n}$, as we expect since the sum of the edgelengths is always 2. In the ensuing sections we will repeatedly use the second moments of edgelength, which we collect in the following corollary.

\begin{corollary}\label{cor:moments}
	The second moments of the distribution of an edgelength $|e_i|$ are
	\begin{align*}
	E(|e_i|^2,\Arm_3(n)) &= \frac{6}{n(n+\nicefrac{1}{2})}, \quad E(|e_i|^2,\Pol_3(n)) = \frac{6}{n(n+1)}, \\
	E(|e_i|^2,\Arm_2(n)) &= \frac{8}{n(n+1)}, \quad E(|e_i|^2,\Pol_2(n)) = \frac{8}{n(n+2)}. 
	\end{align*}
\end{corollary}

There is one more expectation which we will need below. We have written $\Arm_i(n)$ as the union of the spaces $\Arm_i(n,\ell)$ of arms which fail to close by distance $\ell$. We now show:

\begin{proposition}
The expected value of the squared failure to close $\ell^2 = |\sum e_i|^2$ on $\Arm_i(n)$ is given by
\begin{equation*}
E(\ell^2,\Arm_i(n)) = n E(|e_j|^2,\Arm_i(n)).
\end{equation*}
\label{prop:ftc}
\end{proposition}

\begin{proof}
We can expand $\ell^2$ as $\sum |e_i|^2 + 2 \sum \left<e_i,e_j\right>$ and compute the expectation of each term separately. Now
\begin{equation*}
E(\left<e_i,e_j\right>,\Arm_3(n)) = E(\left<\Hopf(q_i),\Hopf(q_j)\right>,S^{4n-1}),
\end{equation*}
and if we write $q = a + b\I + c\J + d\K$, then $q' = c - d\I -a\J + b\K$ has $\Hopf(q') = -\Hopf(q)$. This is easily checked by recalling that \begin{equation*}
\Hopf(a + b\I + c\J + d\K) = (a^2 + b^2 - c^2 - d^2)\I + (2bc - 2ad) \J + (2ac + 2bd) \K.
\end{equation*}
Using this, we see the map $(q_1, \dots, q_i, \dots, q_n) \mapsto (q_1, \dots, q_i', \dots, q_n)$ is an isometry of the quaternionic sphere which reverses the sign of $\left<\Hopf(q_i),\Hopf(q_j)\right>$.  Thus the expectation of $\left<e_i,e_j\right>$ is zero, as desired. The proof for $\Arm_2(n)$ is similar.
\end{proof}
% Note: I know the formula above looks funny, but I've checked it a couple of times.

Since we have nice formulae for the first and second moments of edgelength, we can work out the variance of edgelength with only a bit of algebra:

\begin{corollary}
The variance of edgelength on our spaces is given by
\begin{align*}
V(|e_i|,\Arm_3(n)) &= \frac{2 (n-1)}{n^2 (n+\nicefrac{1}{2})}, \quad V(|e_i|,\Pol_3(n)) = \frac{2 (n-2)}{n^2 (n+1)}. \\
V(|e_i|,\Arm_2(n)) &= \frac{4 (n-1)}{n^2 (n+1)}, \quad V(|e_i|,\Pol_2(n)) = \frac{4 (n-2)}{n^2 (n+2)}.
\end{align*}
\label{cor:variance}
\end{corollary}

Similarly, it is easy to work out the covariance of edgelength.

\begin{corollary}\label{cor:covariance}
	The covariance of edgelength is given by
	\begin{align*}
		\Cov(|e_i|,|e_j|,\Arm_3(n)) &= \frac{-2}{n^2(n+\nicefrac{1}{2})},  \quad \Cov(|e_i|,|e_j|,\Pol_3(n)) = \frac{n-2}{n-1}  \frac{-2}{n^2(n+1)} \\
		\Cov(|e_i|,|e_j|,\Arm_2(n)) &= \frac{-4}{n^2(n+1)},  \quad \Cov(|e_i|,|e_j|,\Pol_2(n)) = \frac{n-2}{n-1}  \frac{-4}{n^2(n+2)}
	\end{align*}
\end{corollary}

\begin{proof}
	We start with the fact that
	\[
		V\left(\sum_{i=1}^n |e_i| \right) = \sum_{i,j} \Cov(|e_i|,|e_j|) = \sum_{i=1}^n V(|e_i|) + \sum_{i \neq j} \Cov(|e_i|,|e_j|).
	\]
	Since $\sum |e_i| = 2$, the left hand side is zero. For $i\neq j$, $\Cov(|e_i|,|e_j|)$ is independent of $i$ and $j$ since we can permute edges; likewise, $V(|e_i|)$ is independent of $i$. Therefore, on each of our polygon spaces the above equation reduces to
	\[
		\Cov(|e_i|,|e_j|) = \frac{-1}{n-1}V(|e_i|).
	\]

	Plugging in the variances from Corollary~\ref{cor:variance} yields the desired covariances.
\end{proof}

We can see that the variance and covariance of edgelength are approaching zero as $n \rightarrow \infty$. It is tempting to think that this makes our polygons ``asymptotically equilateral''. However, Diao~\cite{Diaopc} notes that this is really an artifact of the fact that the edgelength is approaching zero. If we rescale our polygons to length $2n$ so that the mean edgelength is 2, we see that the variances above approach $2$ or $4$ as $n \rightarrow \infty$. Thus our polygons are not becoming ``relatively equilateral'' as $n \rightarrow \infty$: the probability that an edge is larger than a fixed multiple of the mean converges to a positive value. We will compute some of these probabilities in Section~\ref{sec:pdfs}. However, even after rescaling the covariance still goes to zero, so the edgelengths \emph{are} becoming ``asymptotically uncorrelated''.

\section{Averaging Chord Lengths over an Edge Set Ensemble of Polygons}

We now set out to compute the expected value for the chord lengths of a random arm or polygon. Given $(a,b,\theta)$ in the quaternionic sphere which maps to a fixed arm $P = \Hopf(a,b,\theta) = \Hopf(\sqrt{2}a\cos\theta + \sqrt{2}b \sin\theta \J)$, the squared length of the chord skipping the first $k$ edges is
\[
		\chord(k,P) = \left(2 \cos^2\theta \sum_{j=1}^k a_j\bar{a}_j - 2 \sin^2\theta \sum_{j=1}^k b_j\bar{b}_j\right)^2 + 16\sin^2\theta \cos^2\theta \left|\sum_{j=1}^k a_j \bar{b}_j\right|^2.
\]

It does not seem simple to compute the expected value of this formula directly. We now introduce one of the key ideas of this paper: we can improve the situation substantially by averaging the right-hand side over all possible permutations of the edges in order to symmetrize it. Further, the resulting formula will apply to \emph{any} measure on polygon space which is symmetric under rearrangements of an edge set, not just to our measures.

\begin{definition}
We will call the set of polygons obtained by rearranging a set of edges $e_1, \dots, e_n$ the \emph{edge set ensemble} of polygons $\mathcal{P}(e_1, \dots, e_n)$. We note that the sum of the edges is invariant under these rearrangements, so if the polygon $e_1, \dots, e_n$ fails to close, each polygon in the ensemble fails to close by the same vector. Thus we say $\mathcal{P} \in \Arm_i(n,\ell)$ if the polygons in $\mathcal{P}$ fail to close by a vector of length $\ell$.
\label{def:ese}
\end{definition}

We use this to define a new function $\schord(k,\mathcal{P})$. For any measure on polygon space which is invariant under permutations of the edges, this has the same expected value as $\chord(k,P)$.

\clearpage

\begin{definition}
\label{def:schord}
The average of the squared length of the chords skipping the first $k$ edges of polygons in an edge set ensemble $\mathcal{P}$ is
\begin{multline*}
	\schord(k,\mathcal{P}) = \frac{4}{n!} \sum_{\sigma \in S_n} \left[\left(\cos^2\theta\sum_{j=1}^k a_{\sigma(j)}\bar{a}_{\sigma(j)} - \sin^2\theta\sum_{j=1}^k b_{\sigma(j)}\bar{b}_{\sigma(j)}\right)^2 \right. \\  
	\left. + 4\sin^2 \theta \cos^2 \theta \left|\sum_{j=1}^k a_{\sigma(j)} \bar{b}_{\sigma(j)}\right|^2\right].
\end{multline*}
where $S_n$ is the symmetric group on $n$ letters. 
\end{definition}

We can now use some algebra to prove

\begin{proposition}\label{prop:schord}
For any edge set ensemble $\mathcal{P} \in \Arm_i(n,\ell)$,
\[
	 	\schord(k,\mathcal{P}) = \frac{k(n-k)}{n(n-1)}\sum_{j=1}^n |e_j|^2 +		
		 \frac{k(k-1)}{n(n-1)} \ell^2.
\]
\end{proposition}

The proof of this proposition is involved but not terribly illuminating, so we defer it to Appendix~\ref{appendix}. However, the result is surprising and pleasant: the averages of the squared lengths of the chords skipping $k$ edges over all polygons obtained from a given set of $n$ edges depend only on $k$, $n$, and the lengths of the edges! We note that this proposition has nothing to do with our measure on polygon space, so it should be a general tool for computing expected chordlengths for any measure on polygon space. Zirbel and Millett have obtained a similar result independently~\cite{Zirbel:ub}. Figure~\ref{fig:ese_figure} gives a particular example of this theorem.

\begin{figure}[ht]
\begin{tabular}{ccc}
\begin{overpic}[width=1.5in]{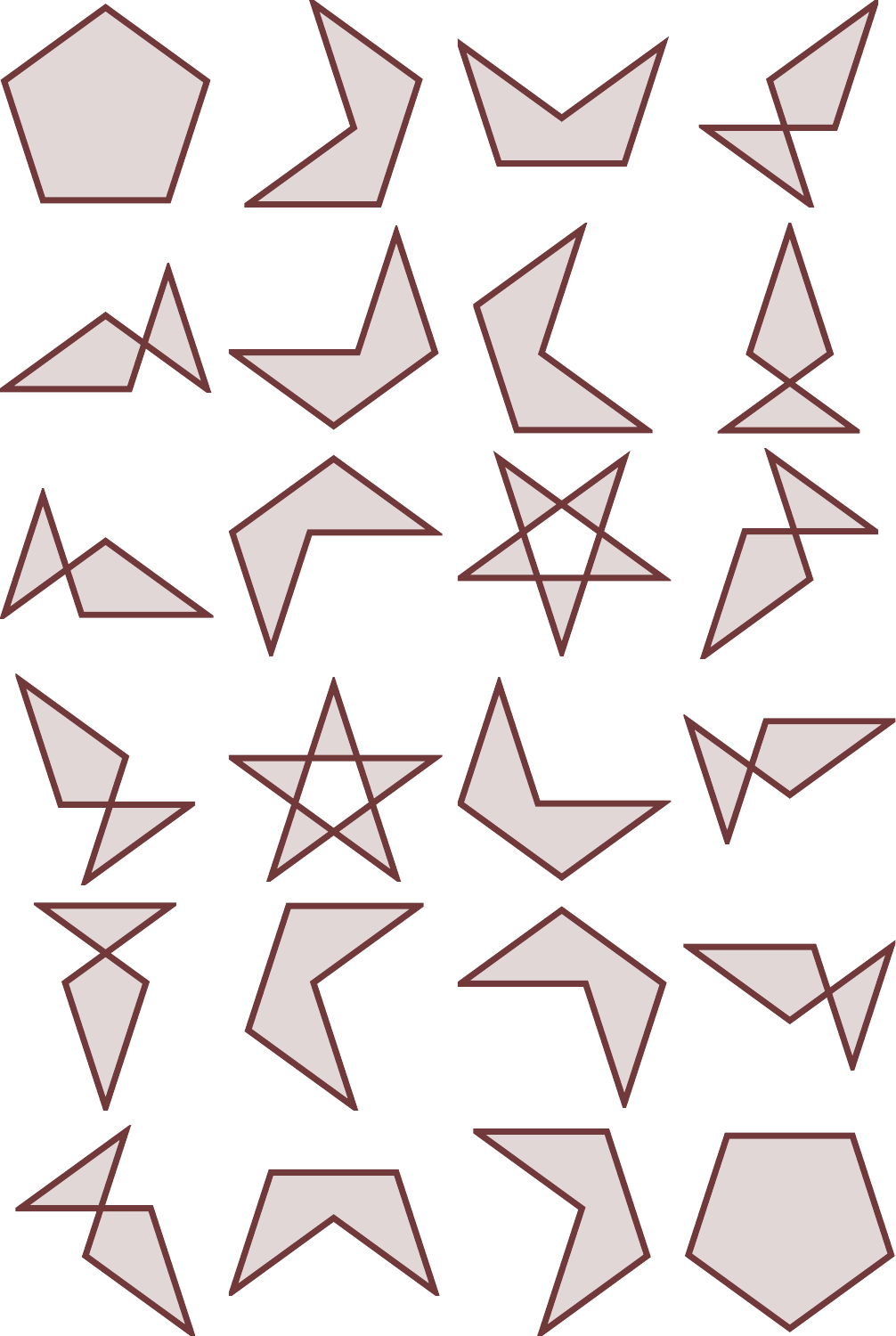}
\end{overpic} 
&
\begin{overpic}[width=1.5in]{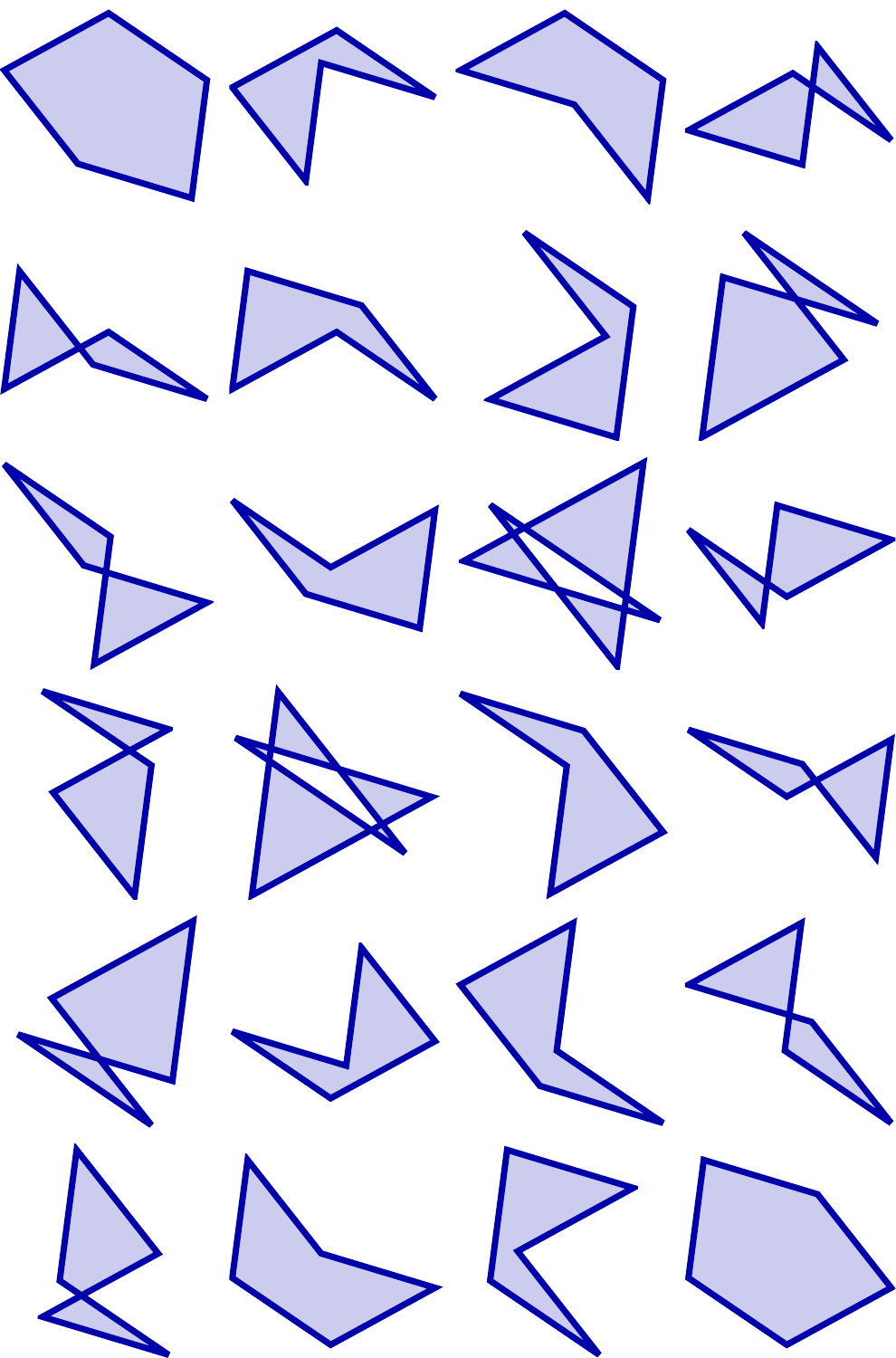}
\end{overpic}
&
\begin{overpic}[width=2.5in]{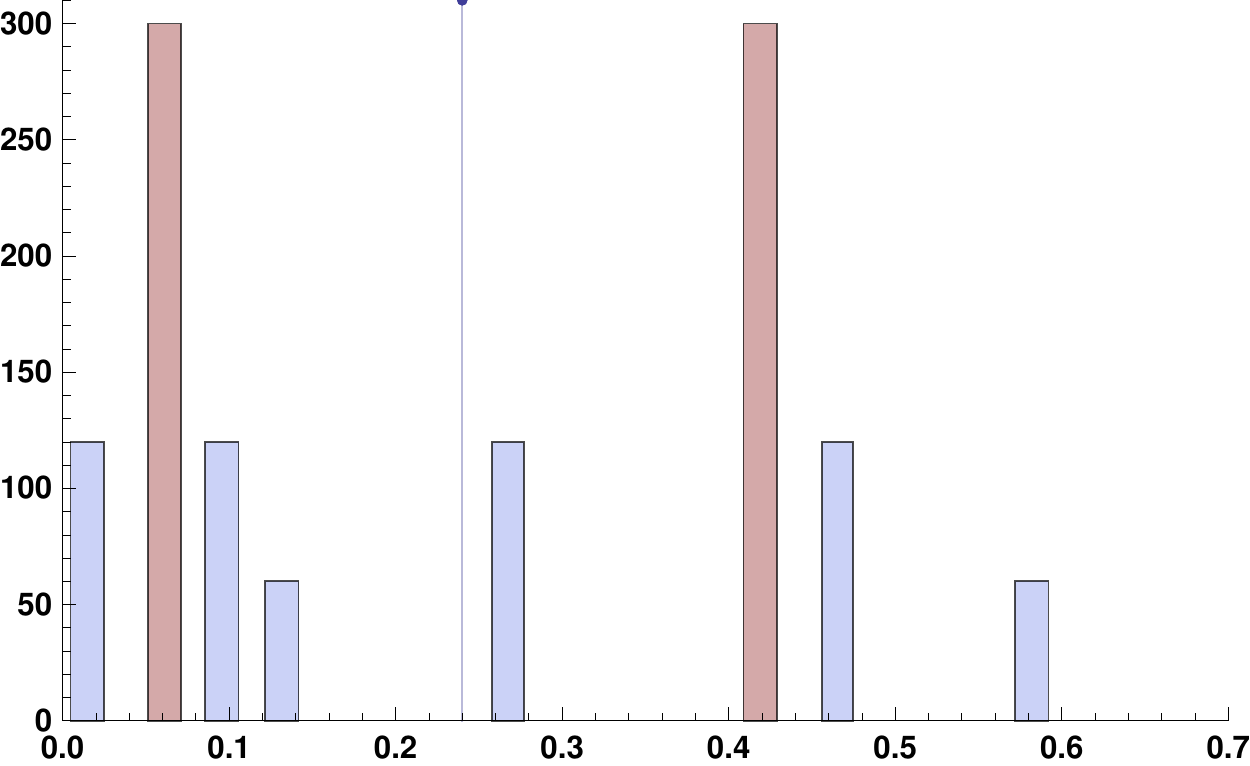}
\end{overpic}
\\
\begin{minipage}{1.5in}
The edge set ensemble including a regular pentagon.
\end{minipage} &
\begin{minipage}{1.5in}
Another edge set ensemble of closed, planar, equilateral 5-gons. 
\end{minipage}&
\begin{minipage}{2.5in}
Histogram of squared chord lengths skipping two edges over both ensembles.
\end{minipage}
\end{tabular}
\caption{The figures above show two edge set ensembles for two sets of 5 edges of length $\nicefrac{2}{5}$. The first set, shown left, includes the regular pentagon of length 2. The second set, shown center, is irregular. Though each edge set ensemble contains $5! = 120$ polygons, we reduce by cyclic reordering to display only 24 representative polygons from each. The 120 polygons in each set have a total of $120 \cdot 5 = 600$ chords which skip two edges. The right-hand plot shows a histogram of the squared lengths of these 600 chords for the regular edge set ensemble (pair of tall darker bars) and the irregular edge set ensemble (set of six shorter bars). Since the polygons close, $\ell = 0$, and Proposition~\ref{prop:schord} predicts a mean of $\nicefrac{6}{25}$ for both data sets. This value is indicated by a thin vertical line. The mean value of each data set agrees with $\nicefrac{6}{25}$ to at least 15 digits in \emph{Mathematica}.\label{fig:ese_figure}}
\end{figure}

We can also consider the radius of gyration of our polygons. If $P$ is assembled from the edges $e_1, \dots, e_n$ (in order) then the radius of gyration is half of the average squared distance between any two vertices of the polygon (including repeated pairs where the distance is zero). Again, we can symmetrize this definition:
\begin{definition}
The radius of gyration of an open polygon $P$ with $n$ edges $e_1, \dots, e_n$ and $n+1$ vertices $v_1, \dots, v_{n+1}$ is 
\begin{equation*}
\gyradius(e_1, \dots, e_n) = \frac{1}{2(n+1)^2} \sum_{i,j} \left| v_i - v_j \right|^2.
\end{equation*}
We can symmetrize this formula over an edge set ensemble $\mathcal{P}$ to get
\begin{equation*}
\sgyradius(\mathcal{P}) = \frac{1}{n!} \sum_{\sigma \in S_n} \gyradius(e_{\sigma(1)}, \dots, e_{\sigma(n)})
\end{equation*}
\end{definition}

We can then prove

\begin{proposition}
\label{prop:sgyradius}
For any edge set ensemble of polygons $\mathcal{P} \in \Arm_i(n,\ell)$ with $\ell > 0$, 
\begin{equation*}
	\sgyradius(\mathcal{P}) = \frac{n+2}{12(n+1)} \left( \sum_{j=1}^n |e_j|^2 + \ell^2 \right).
\end{equation*}
We think of an edge set ensemble of polygons $\mathcal{P} \in \Pol_i(n) = \Arm_i(n,0)$ as having one fewer vertex. Thus, the correct formula for $\sgyradius$ is 
\begin{equation*}
	\sgyradius(\mathcal{P}) = \frac{n+1}{12n} \left( \sum_{j=1}^n |e_j|^2 \right).
\end{equation*}
\end{proposition}

\begin{proof}
We start with some algebra on the definition of $\gyradius$:
\begin{multline*}
	\gyradius(e_1, \dots, e_n) = \frac{1}{2(n+1)^2} \sum_{i,j} \left| v_i - v_j \right|^2 = \frac{1}{(n+1)^2}  \sum_{j=1}^n \sum_{i=j+1}^{n+1} \left| v_i - v_j \right|^2 \\
	= \frac{1}{(n+1)^2}  \sum_{j=1}^n \sum_{k=1}^{n+1-j} \left| v_{j+k} - v_j \right|^2 
	         =  \frac{1}{(n+1)^2}  \sum_{j=1}^{n} \sum_{k=1}^{n+1-j} \left| \sum_{m = j}^{j+k-1} e_m \right|^2. 
\end{multline*}
We can then continue with $\sgyradius$:
\begin{align*}
   	\sgyradius(\mathcal{P}) &= \frac{1}{n!} \sum_{\sigma \in S_n} \gyradius(e_{\sigma(1)}, \dots, e_{\sigma(n)}) \\
						   &= \frac{1}{n!} \sum_{\sigma \in S_n} \left( \frac{1}{(n+1)^2}  \sum_{j=1}^{n} \sum_{k=1}^{n+1-j} \left| \sum_{m = j}^{j+k-1} e_{\sigma(m)} \right|^2 \right)\\
	                        &= \frac{1}{(n+1)^2}  \sum_{j=1}^{n} \sum_{k=1}^{n+1-j} \frac{1}{n!} \sum_{\sigma \in S_n} \left| \sum_{m = j}^{j+k-1} e_{\sigma(m)} \right|^2 \\
	&= \frac{1}{(n+1)^2} \sum_{j=1}^n \sum_{k=1}^{n+1-j} \schord(k,\mathcal{P}) \\
	&= \sum_{k=1}^n \frac{(n+1)-k}{(n+1)^2} \schord(k,\mathcal{P}).
\end{align*}
Combining this with Proposition~\ref{prop:schord} and summing over $k$ completes the proof.
\end{proof}

\section{Expected Value of Chord Length and Radius of Gyration}

We can now use our previous results on the moments of the edgelength distribution to compute the expected value of squared chord length. We note that combining Proposition~\ref{prop:schord} with our expected value for $\ell^2$ from Proposition~\ref{prop:ell} and simplifying the coefficients yields that the expected value of $\chord(k)$ on arm space is given by
\begin{equation*}
E(\chord(k),\Arm_i(n)) = k E(|e_j|^2,\Arm_i(n)).
\end{equation*}
On polygon space, the situation is similar:
\begin{equation*}
E(\chord(k),\Pol_i(n)) = \left( \frac{n-k}{n-1} \right) \, k E(|e_j|^2,\Pol_i(n)).
\end{equation*}
We can now compute explicit formulae for these expectations:

\begin{proposition}
The expected values for $\chord(k)$ on the arm and polygon spaces are:
\begin{align*}
E(\chord(k),\Arm_3(n)) &= \frac{6 k}{n(n+\nicefrac{1}{2})}, \quad E(\chord(k),\Pol_3(n)) = \left(\frac{n-k}{n-1}\right) \frac{6k}{n(n+1)}. \\
E(\chord(k),\Arm_2(n)) &= \frac{8 k}{n(n+1)}, \quad E(\chord(k),\Pol_2(n)) = \left( \frac{n-k}{n-1} \right) \frac{8k}{n(n+2)}.
\end{align*}
\label{prop:expected chords}
\end{proposition}
 
\begin{proof}
This follows directly from our observations above, Proposition~\ref{prop:schord}, and Corollary~\ref{cor:moments}.
\end{proof}
 
We can use Proposition~\ref{prop:sgyradius} and our work above to compute expected values for radius of gyration as well since again (by construction) the expected values of $\sgyradius$ match those of $\gyradius$:
\begin{proposition}
\label{prop:egyradius}
The expected value of $\gyradius$ over our arm and polygon spaces is
\begin{align*}
E(\gyradius,\Arm_3(n)) &=  \frac{n+2}{(n+1)(n+\nicefrac{1}{2})}, \quad E(\gyradius,\Pol_3(n)) = \frac{1}{2} \frac{1}{n}, \\
E(\gyradius,\Arm_2(n)) &= \frac{4}{3} \frac{n+2}{(n+1)^2}, \quad E(\gyradius,\Pol_2(n)) = \frac{2}{3} \frac{n+1}{n(n+2)}.
\end{align*}
\end{proposition}

We can use these chordlength formulae to compute the expected value of the inner product of two edge vectors in a polygon:
\begin{corollary}
The expected value of $\left< e_i, e_j \right>$ on our spaces is given by
\begin{align*}
E(\left<e_i,e_j\right>,\Arm_3(n)) &= 0, \quad E(\left<e_i,e_j\right>,\Pol_3(n)) = -\frac{6}{n^3 - n}. \\
E(\left<e_i,e_j\right>,\Arm_2(n)) &= 0, \quad E(\left<e_i,e_j\right>,\Pol_2(n)) = -\frac{8}{(n-1)n(n+2)}.
\end{align*}
\label{cor:dotproduct}
\end{corollary}

\begin{proof}
We computed the expected value on the $\Arm_i$ spaces above in the proof of Proposition~\ref{prop:ftc}. For the $\Pol_i$ spaces, observe that 
\begin{equation*}
\left< e_i + e_j, e_i + e_j \right> = \left< e_i, e_i\right> + 2 \left<e_i,e_j\right> + \left<e_j,e_j\right>
\end{equation*}
Taking the expected value of each side of the equation and rearranging, we get
\begin{equation*}
E(\left<e_i,e_j\right>,\Pol_d(n)) = \frac{1}{2} E(\chord(2),\Pol_d(n)) - E(\chord(1,\Pol_d(n)),
\end{equation*}
which leads directly to the formulae above.
\end{proof}
We note that this formula, together with the computations of covariances of edgelengths in Corollary~\ref{cor:covariance}, gives us an explicit calculation of the pairwise correlations between edges.

\section{Equilateral Polygon Space}
In our theory, the space of equilateral polygonal arms or closed polygons have a special place: they are the only fixed edge-length polygon spaces which are invariant under rearrangement of edges. Our symmetric measure restricts to this space (as a subspace of $\Arm$ or $\Pol$) as the product of measures on round spheres (of dimension 1 or 2) or the subspace of subsets of the product of spheres which sum to zero, and the $\schord$ and $\sgyradius$ formulae of Propositions~\ref{prop:schord} and~\ref{prop:sgyradius} apply as well. Since the expected value of the edgelengths is easy to compute, we immediately get expected values for chord lengths and radius of gyration. Interestingly, they do not depend on the ambient dimension, since the expected value of edgelength is the same in each case.

\begin{proposition}
The expected values of $\chord(k)$ on equilateral arm and polygon spaces is 
\begin{align*}
E(\chord(k),\eArm_i(n)) &= \frac{4k}{n^2}, \quad E(\chord(k),\ePol_i(n)) = \left(\frac{n-k}{n-1}\right)\frac{4k}{n^2}. 
\end{align*}
\label{prop:equilaterals}
\end{proposition}
It is interesting to compare this to the mean of the approximate pdf $h_k(r)$ for the $k$th chord length $r$ of a (closed) equilateral space polygon of length $n$ given in \cite{Varela:2009cd}:
\begin{equation*}
h_k(r) = \left( \frac{3}{2\pi \frac{k(n-k)}{n}} \right)^{\frac{3}{2}} 4 \pi r^2 e^{-\frac{3}{2} \frac{n}{k(n-k)} r^2}.
\end{equation*}
We compute that $\int_0^\infty r^2 h_k(r) \dr = \frac{k(n-k)}{n}$,
so that (rescaling appropriately) the mean of $r^2$ with respect to the pdf $h_k(r)$ is $\frac{k(n-k)}{n} \frac{4}{n^2}$, which converges to the result of Proposition~\ref{prop:equilaterals}.

Similarly, we can compute the expected value of radius of gyration:
\begin{proposition}
The expected value of $\gyradius$ on equilateral arm and polygon spaces is
\begin{equation*}
E(\gyradius,\eArm_i(n)) = \frac{2}{3} \frac{n+2}{n(n+1)}, \quad E(\gyradius,\ePol_i(n)) = \frac{1}{3} \frac{n+1}{n^2}.
\end{equation*}
\label{prop:eqgyradius}
\end{proposition}
As in Corollary~\ref{cor:dotproduct} we can compute the expected value of the inner product of two edges as
\begin{corollary}
The expected value of the inner product $\left< e_i, e_j \right>$ on equilateral arm and polygon spaces is 
\begin{equation*}
E(\left< e_i, e_j \right>,\eArm_i(n)) = 0, \quad E(\left< e_i, e_j \right>,\ePol_i(n)) = -\frac{4}{n^2}\frac{1}{n-1}.
\end{equation*}
\label{cor:edotproduct}
\end{corollary}
This last formula (accounting for the fact that our polygons are scaled differently) recovers the formula of Grosberg~\cite{Grosberg:2008dl} for the expected value of the inner product for equilateral closed polygons. The slightly negative expectation seems to come from the fact that for any edge $e_i$ of a closed $n$-gon, the other $n-1$ edges must add up to $-e_i$ in order to close the polygon. Corollary~\ref{cor:dotproduct} shows the same phenomenon for our larger polygon spaces. 

We saw before in Corollary~\ref{cor:variance} that the variance of edgelength is going to zero as $n \rightarrow \infty$ and the expected value of edgelength goes to zero. However, if we rescale our polygons to length $2n$, the variance of edgelength goes to 2 for (closed) space polygons and 4 for (closed) plane polygons-- that is, the \emph{relative} variance of edgelength does not go to zero. Thus, while the expected value of $\gyradius$ goes to zero as $n \rightarrow \infty$ for equilateral polygons as well as our original polygons, the expected value of $\gyradius$ for equilateral rescaled polygons shouldn't converge to the corresponding expected value of $\gyradius$ for our original polygons rescaled to length $2n$. In fact, they don't: since $\gyradius$ scales quadratically with length, scaling to length $2n$ gives us an expected $\gyradius$ of $\nicefrac{4}{3} (n + 1)$ for (space or plane) equilateral polygons. On the other hand, Proposition~\ref{prop:egyradius} tells us that the corresponding expectation for space polygons is $4n$ while the expectation for plane polygons is $\frac{8}{3} \frac{n(n+1)}{n+2}$.

\section{Edgelength Distributions}
\label{sec:pdfs}

We can say more precisely how far our polygons are from being equilateral by analyzing the probability density function of edgelength. Note that the $i$th edgelength of an arm can be anything from $0$ to $2$, so the corresponding probability density function is a function on $[0,2]$, whereas polygon edgelengths must be $\leq 1$, so the probability density function for polygons has domain $[0,1]$. 

\begin{proposition}\label{prop:pdfs}
	The probability density functions for the edgelength $|e_i|$ of spatial and planar arms are given by
	\[
		\phi_3^\mathrm{Arm} (y) = \frac{(2n-2)(2n-1)}{2^{2n-1}}y(2-y)^{2n-3},\qquad \phi_2^\mathrm{Arm} (y) = \frac{n-1}{2^{n-1}}(2-y)^{n-2}
	\]
	and the probability density functions for the edgelength $|e_i|$ of spatial and planar polygons are
	\[
		\phi_3^\mathrm{Pol} (y) = (n-2)(n-1)y(1-y)^{n-3},\qquad \phi_2^\mathrm{Pol} (y) = (\nicefrac{n}{2}-1)(1-y)^{\nicefrac{n}{2}-2}.
	\]
\end{proposition}

\begin{figure}[htbp]
\centering
\subfloat[Arms in space]{\begin{tabular}{ccccc}
\includegraphics[height=.7in]{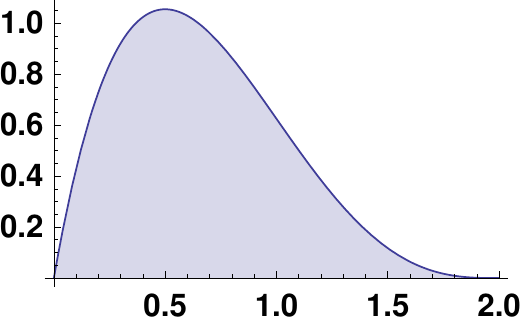} & 
\includegraphics[height=.7in]{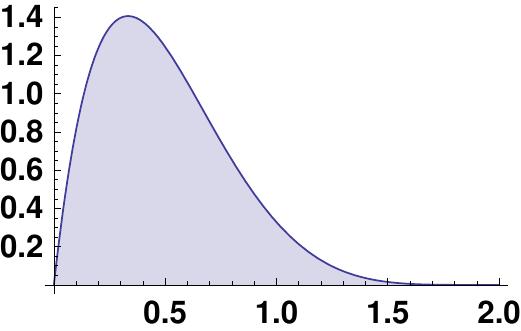} & 
\includegraphics[height=.7in]{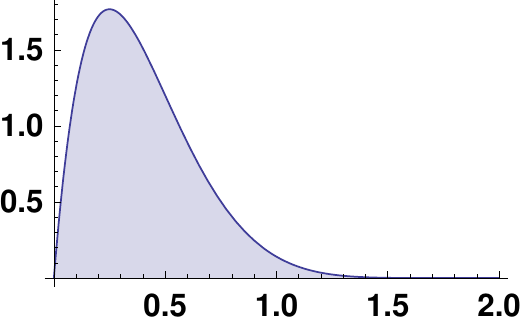} & 
\includegraphics[height=.7in]{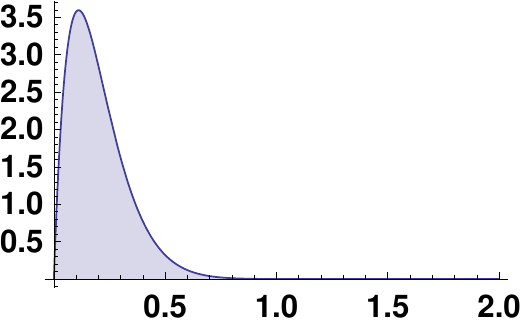} & 
\includegraphics[height=.7in]{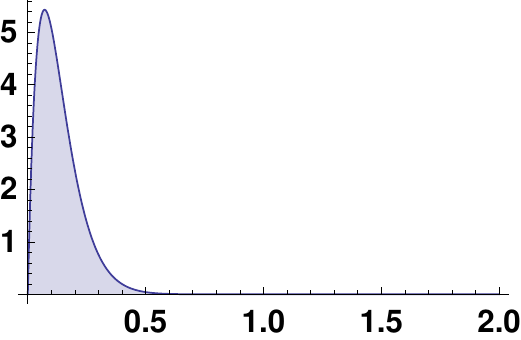} \\
{\footnotesize 3 edges} & {\footnotesize 4 edges} & {\footnotesize 5 edges} & {\footnotesize 10 edges} & {\footnotesize 15 edges}
\end{tabular}}

\subfloat[Arms in the plane]{\begin{tabular}{ccccc}
\includegraphics[height=.7in]{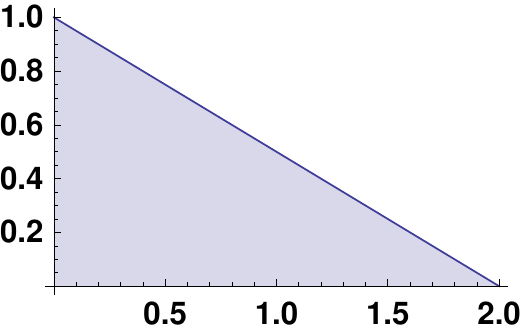} & 
\includegraphics[height=.7in]{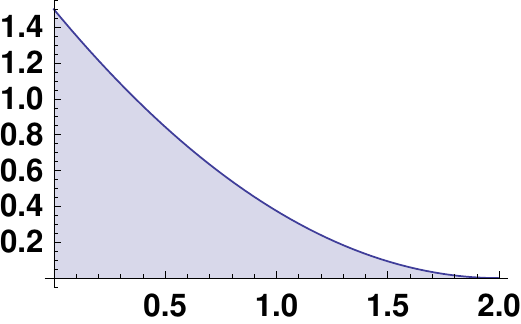} & 
\includegraphics[height=.7in]{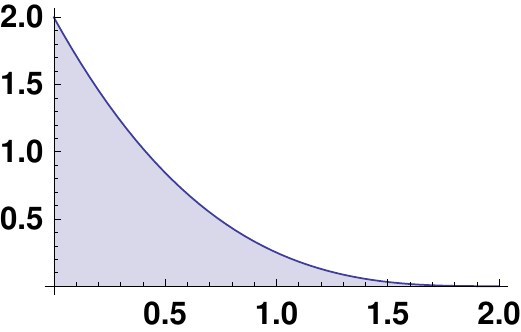} & 
\includegraphics[height=.7in]{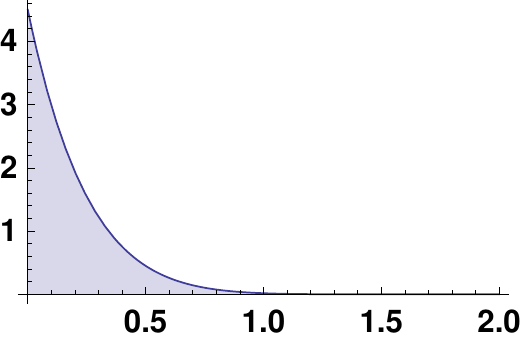} & 
\includegraphics[height=.7in]{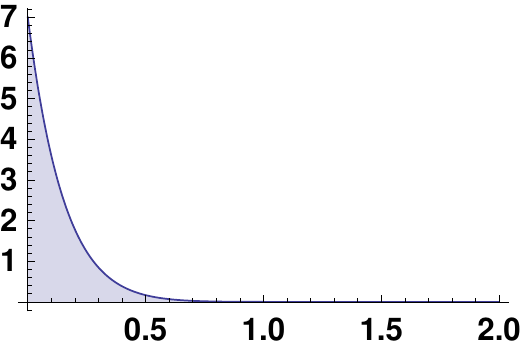} \\
{\footnotesize 3 edges} & {\footnotesize 4 edges} & {\footnotesize 5 edges} & {\footnotesize 10 edges} & {\footnotesize 15 edges}
\end{tabular}}

\subfloat[Polygons in space]{\begin{tabular}{ccccc}
\includegraphics[height=.7in]{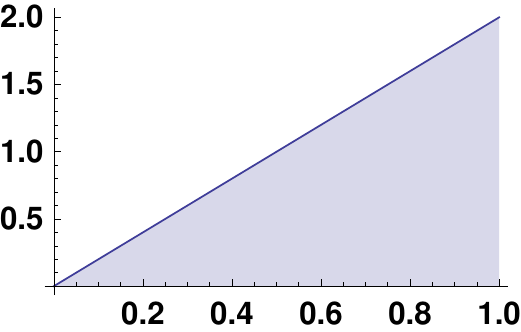} & 
\includegraphics[height=.7in]{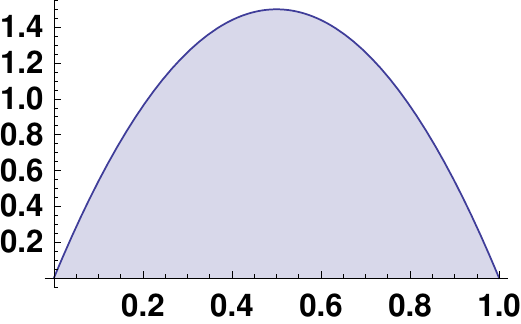} & 
\includegraphics[height=.7in]{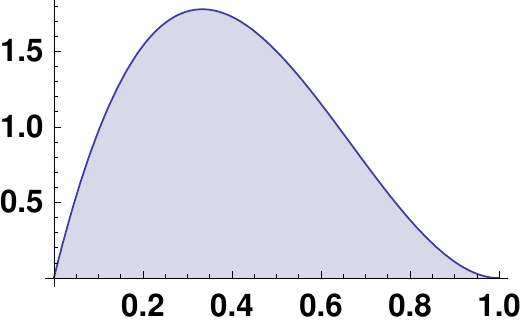} & 
\includegraphics[height=.7in]{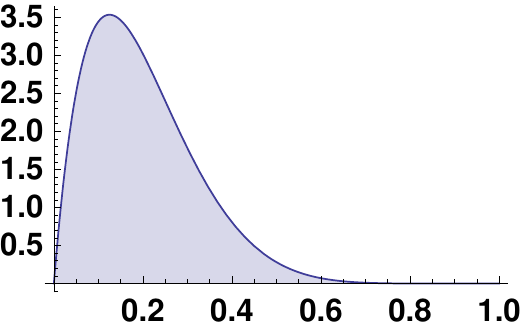} & 
\includegraphics[height=.7in]{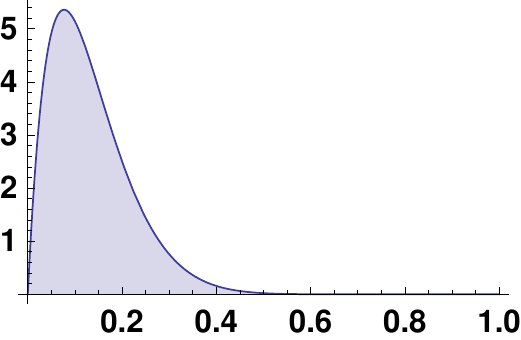} \\
{\footnotesize 3 edges} & {\footnotesize 4 edges} & {\footnotesize 5 edges} & {\footnotesize 10 edges} & {\footnotesize 15 edges}
\end{tabular}}

\subfloat[Polygons in the plane]{\begin{tabular}{ccccc}
\includegraphics[height=.7in]{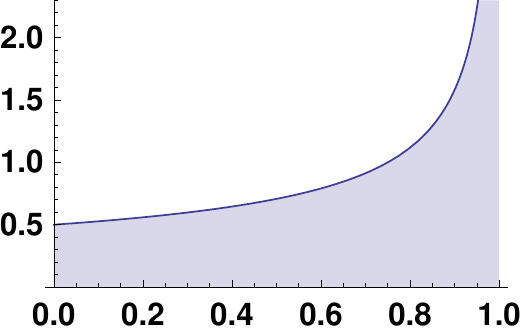} & 
\includegraphics[height=.7in]{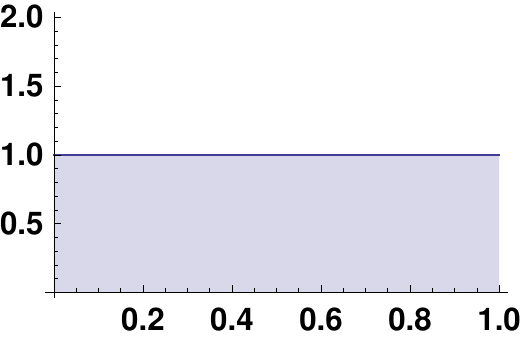} & 
\includegraphics[height=.7in]{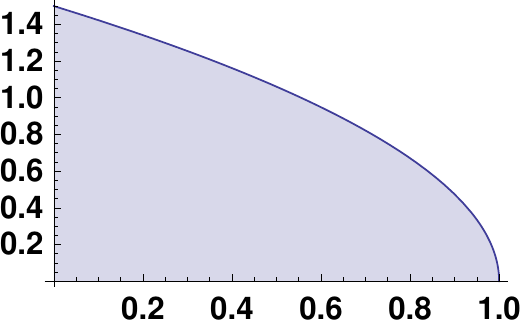} & 
\includegraphics[height=.7in]{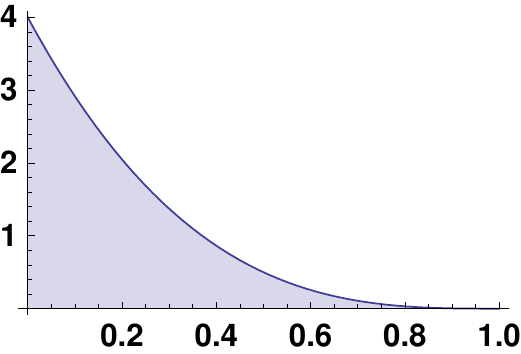} & 
\includegraphics[height=.7in]{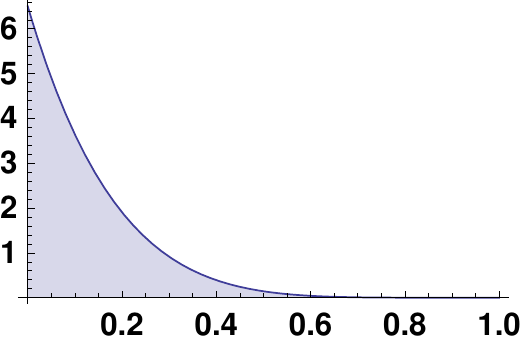} \\
{\footnotesize 3 edges} & {\footnotesize 4 edges} & {\footnotesize 5 edges} & {\footnotesize 10 edges} & {\footnotesize 15 edges}
\end{tabular}}

\caption{Graphs of the probability density functions for the $i$th edgelength of arms and polygons in space and in the plane. 
\label{fig:pdfs}}
\end{figure}

Graphs of these functions for a few values of $n$ are shown in Figure~\ref{fig:pdfs}. Notice that each of the above is the probability density function of a beta distribution on the appropriate domain, as summarized in the following table:

\begin{center}
\begin{tabular*}{0.5\textwidth}{@{\extracolsep{\fill}}lcc}
	  & \multicolumn{2}{c}{Beta distribution with} \\
	& domain &  shape parameters \\
	\hline
	$\phi_3^\mathrm{Arm}(y)$ & $[0,2]$ & $2$, $2n-2$ \\
	$\phi_2^\mathrm{Arm}(y)$ & $[0,2]$ & $1$, $n-1$ \\
	$\phi_3^\mathrm{Pol}(y)$ & $[0,1]$ & $2$, $n-2$ \\
	$\phi_2^\mathrm{Pol}(y)$ & $[0,1]$ & $1$, $\nicefrac{n}{2} - 1$\\
\end{tabular*}
\end{center}

The appearance of beta distributions may seem somewhat surprising, but this is actually to be expected: given the interval $[0,2]$ with $n-1$ points chosen uniformly on it (which we might think of as an ``abstract'' arm), the complement of those points consists of (generically) $n$ subintervals whose lengths follow a beta distribution (the argument below is standard and can be found in, e.g., Feller~\cite{Feller}).

More specifically, the length $L_1$ of the first subinterval is greater than $y$ if and only if all of the $n-1$ chosen points lie in the interval $(y,2]$, which happens with probability $\frac{1}{2^{n-1}}(2-y)^{n-1}$. Therefore, the cumulative distribution function of $L_1$ is $1-\frac{1}{2^{n-1}}(2-y)^{n-1}$ and so the probability density function is the derivative
\begin{equation}\label{eqn:abstractPDF}
	\phi(y) = \frac{n-1}{2^{n-1}}(2-y)^{n-2}.
\end{equation}
Moreover, the lengths of all of the $n$ subintervals follow the same distribution, as can be seen by noticing that choosing $n-1$ points uniformly on $[0,2]$ is the same as choosing $n$ points uniformly on a circle of circumference $2$ and cutting at the first point.

Notice that the pdf in \eqref{eqn:abstractPDF} is exactly the probability density function of edgelength for planar arms given in Proposition~\ref{prop:pdfs}, so in our model for planar arms the edgelength distributions are insensitive to how the arms lie in the plane. Of course the same \emph{cannot} be true for polygons, and indeed the edgelength distributions for our polygon spaces are different from the above distribution. Perhaps more surprising is that in our model for spatial arms the edgelength distribution does not match the above.

\begin{proof}[Proof of Proposition~\ref{prop:pdfs}]
	Define the map $F_i: S^{4n-1} \to [0,2]$ by $(q_1, \ldots , q_n) \mapsto |q_i|^2$, which is just the length of the $i$th edge of the corresponding arm (recall that, as in Section~\ref{sec:armSpace}, the total space lying over $\Arm_3(n)$ is the sphere of radius $\sqrt{2}$ in $\Q^n$). Since the probability density function on $S^{4n-1}$ is uniform with respect to the standard metric, we can compute the corresponding probability density function $\phi_3^\mathrm{Arm}$ on $[0,2]$ using the coarea formula as
	\begin{equation}\label{eq:pdfv1}
		\phi_3^\mathrm{Arm}(y) = \frac{1}{\Vol S^{4n-1}(\sqrt{2})} \int_{\vec{q} \in F_i^{-1}(y)} \frac{1}{|\nabla F_i(\vec{q})|} \dVol_{F_i^{-1}(y)},
	\end{equation}
	where $\nabla$ indicates the intrinsic gradient in $S^{4n-1}$, $F_i^{-1}(y)$ is the hypersurface corresponding to arms with $i$th edge of length $y$ and the measure $\dVol_{F_i^{-1}(y)}$ comes from the subspace metric on this space as a submanifold of $S^{4n-1}$. 
	
	If $q_i = q_{i,1} + q_{i,2}\I + q_{i,3}\J + q_{i,4}\K$, then $F_i(\vec{q}) = q_{i,1}^2 + q_{i,2}^2 + q_{i,3}^2 + q_{i,4}^2$ and so the extrinsic gradient of $F_i$ is
	\[
		\nabla_{\Q^n}F_i(\vec{q}) = 2q_{i,1}\vec{x}_{i,1} + 2q_{i,2}\vec{x}_{i,2} + 2q_{i,3}\vec{x}_{i,3} + 2q_{i,4}\vec{x}_{i,4}.
	\]
	Since the unit normal to $\vec{q}$ is just $\vec{n} = \frac{\vec{q}}{\sqrt{2}}$, we see that $\langle \vec{n},\nabla_{\Q^n}F_i(\vec{q})\rangle = \sqrt{2}F_i(\vec{q})$, and so
	\[
		|\nabla F_i(\vec{q})| = \sqrt{|\nabla_{\Q^n}F_i(\vec{q})|^2 - \langle \vec{n},\nabla_{\Q^n}F_i(\vec{q})\rangle^2} = \sqrt{2}\sqrt{F_i(\vec{q})}\sqrt{2-F_i(\vec{q})}.
	\]
	Since $|\nabla F_i(\vec{q})|$ is constant on $F_i^{-1}(y)$, equation \eqref{eq:pdfv1} simplifies as
	\begin{equation}\label{eq:pdfv2}
		\phi_3^\mathrm{Arm}(y) = \frac{1}{\sqrt{2}\sqrt{y}\sqrt{2-y}} \cdot \frac{\Vol F_i^{-1}(y)}{\Vol S^{4n-1}(\sqrt{2})}.
	\end{equation}
	
	If $F_i(\vec{q}) = |q_i|^2 = y$, then $q_i$ lives on a $3$-sphere of radius $\sqrt{y}$, while the remaining $n-1$ quaternionic coordinates lie on a $(4n-5)$-sphere of radius $\sqrt{2-y}$. Therefore,
	\[
		\Vol F_i^{-1}(y) = \left(\Vol S^3\left(\sqrt{y \vphantom{2}}\right)\right) \left(\Vol S^{4n-5}\left(\sqrt{2-y}\right)\right) = 2\pi^2 \sqrt{y}^3 \sqrt{2-y}^{4n-5} \Vol S^{4n-5}(1).
	\]
	Combining this with \eqref{eq:pdfv2} yields
	\[
		\phi_3^\mathrm{Arm}(y) = \sqrt{2}\pi^2y(2-y)^{2n-3} \cdot \frac{\Vol S^{4n-5}(1)}{\Vol S^{4n-1}(\sqrt{2})} = \frac{(2n-2)(2n-1)}{2^{2n-1}}y(2-y)^{2n-3},
	\]
	as desired.
	
	Completely analogous reasoning yields the probability density function for planar arms since the intrinsic gradient of edgelength is exactly the same in this case.
	
	For space polygons, the map $F_i$ defined above induces the $i$th edgelength map $F_i: V_2(\C^n) \to [0,1]$ which is given by $(a,b) \mapsto |a_i|^2 + |b_i|^2$. Using the coarea formula, the probability density function on $[0,1]$ is
	\begin{equation}\label{eq:polpdfv1}
		\phi_3^\mathrm{Pol}(y) = \frac{1}{\Vol V_2(\C^n)}\int_{(a,b) \in F_i^{-1}(y)} \frac{1}{|\nabla F_i(a,b)|} \dVol_{F_i^{-1}(y)}.
	\end{equation}
	
	The rest of the argument proceeds as with arms, with the following variations. First, a slight modification of the earlier calculation shows that, for $(a,b) \in F_i^{-1}(y)$, the edgelength function $F_i$ has intrinsic gradient
	\[
		|\nabla F_i(a,b)| = 2 \sqrt{y}\sqrt{1-y}
	\]
	as Jianwei~\cite{Jianwei:2006wl} also observed in the process of calculating the homology of real and complex Grassmannians using a (degenerate) Morse function which reduces to $F_i$ for Grassmannians of $2$-planes. Therefore, \eqref{eq:polpdfv1} simplifies as
	\begin{equation}\label{eq:polpdfv2}
		\phi_3^\mathrm{Pol}(y) = \frac{1}{2 \sqrt{y}\sqrt{1-y}}\cdot \frac{\Vol F_i^{-1}(y)}{\Vol V_2(\C^n)}.
	\end{equation}
	
	The other major difference from arms is in the volume of $F_i^{-1}(y)$, which is given by the following lemma.

\begin{lemma}\label{lem:volumeOfEdgeSpace}
	The hypersurface in $V_2(\C^n)$ corresponding to spatial $n$-gons with $i$th edge of length~$y$ has $(4n-5)$-dimensional volume
	\begin{equation}\label{eq:propVolumeSpatial} 
		\Vol F_i^{-1}(y) = 2\pi^2 \sqrt{y}^3 \sqrt{1-y}^{2n-5} \Vol \, V_2(\C^{n-1}).
	\end{equation}
			
	Similarly, the hypersurface in $V_2(\R^n)$ corresponding to planar $n$-gons with $i$th edge of length~$y$ has $(2n-4)$-dimensional volume
	\begin{equation}\label{eq:propVolumePlanar}
		2\pi \sqrt{y} \sqrt{1-y}^{n-3} \Vol\, V_2(\R^{n-1}).
	\end{equation}
\end{lemma}

We defer the proof of this lemma for the moment and observe that \eqref{eq:propVolumeSpatial} allows us to simplify \eqref{eq:polpdfv2} as
\[
	\phi_3^\mathrm{Pol}(y) = \pi^2 y (1-y)^{n-3} \frac{\Vol V_2(\C^{n-1})}{\Vol V_2(\C^n)}= (n-2)(n-1)y(1-y)^{n-3}
\]
since $\frac{\Vol V_2(\C^{n-1})}{\Vol V_2(\C^n)} = \frac{(n-2)(n-1)}{\pi^2}$.

The argument for planar polygons is entirely parallel.
\end{proof}

\begin{proof}[Proof of Lemma~\ref{lem:volumeOfEdgeSpace}]
	Since the map $F_i: V_2(\C^n) \to [0,1]$ is given by $(a,b) \mapsto |a_i|^2 + |b_i|^2$, it is $U(2)$-invariant and so descends to a map $\widehat{F}_i: G_2(\C^n) \to [0,1]$. 
	
	As Jianwei observes \cite[Theorem~3.1]{Jianwei:2006wl}, the inverse image $\widehat{F}_i^{-1}(0)$ is a copy of $G_2 (\C^{n-1})$ and $\widehat{F}_i^{-1}(1)$ is a copy of $G_1 (\C^{n-1}) = \C P^{n-2}$. Moreover, following the reasoning from~\cite[Section~4.1]{2009arXiv0909.1967S}, the hypersurface $\widehat{F}_i^{-1}(y)$ for $y \in (0,1)$ is homeomorphic to $V_2(\C^n)/U(1)$, which can be viewed as an $SU(2)$-bundle over $\widehat{F}_i^{-1}(0) \simeq G_2(\C^{n-1})$ and as an $S^{2n-5}$-bundle over $\widehat{F}_i^{-1}(1) \simeq \C P^{n-2}$.
	
	Geometrically, the fiber over $\widehat{F}_i^{-1}(0)$ is scaled by $\sqrt{y}$, whereas the fiber over $\widehat{F}_i^{-1}(1)$ is scaled by $\sqrt{1-y}$. Therefore, as a hypersurface in $G_2(\C^n)$,
	\[
		\Vol \widehat{F}_i^{-1}(y) = \sqrt{y}^3 \sqrt{1-y}^{2n-5} \Vol\left(V_2(\C^{n-1})/U(1)\right) = \frac{1}{2\pi} \sqrt{y}^3 \sqrt{1-y}^{2n-5} \Vol V_2(\C^{n-2}).
	\]
	The hypersurface $F_i^{-1}(y) \subset V_2(\C^n)$ is just the Stiefel fiber over $\widehat{F}_i^{-1}(y) \subset G_2(\C^n)$; since the fibers of the Stiefel projection $V_2(\C^n) \to G_2(\C^n)$ are standard copies of $U(2)$, we can just multiply the right hand side of the above by $\Vol U(2) = 4\pi^3$ to get
	\[
		\Vol F_i^{-1}(y) = 2\pi^2 \sqrt{y}^3 \sqrt{1-y}^{2n-5} \Vol V_2(\C^{n-1}),
	\]
	as desired. 
	
	Re-using notation, let $F_i: V_2(\R^n) \to [0,1]$ be given by $(a,b) \mapsto |a_i|^2 + |b_i|^2$, which descends to $\widehat{F}_i: G_2(\R^n) \to [0,1]$. 
	
	Then \eqref{eq:propVolumePlanar} says that, as a hypersurface in $V_2(\R^n)$,
	\[
		\Vol F_i^{-1}(y)  = 2\pi \sqrt{y} \sqrt{1-y}^{n-3} \Vol\, V_2(\R^{n-1}).
	\]
	Essentially the same argument as in the complex case works: for $y \in (0,1)$ the inverse image $\widehat{F}_i^{-1}(y) \subset G_2(\R^n)$ is homeomorphic to $V_2(\R^{n-1})/O(1)$ and geometrically is an $SO(2)$-bundle over $G_2(\R^{n-1})$ with fibers scaled by $\sqrt{y}$, as well as an $S^{n-3}$-bundle over $G_1(\R^{n-1}) = \R P^{n-2}$ with fibers scaled by $\sqrt{1-y}$. Therefore, as a hypersurface in $G_2(\R^n)$ we have
	\[
		\Vol \widehat{F}_i^{-1}(y) = \sqrt{y} \sqrt{1-y}^{n-3} \Vol\left(V_2(\R^{n-1})/O(1)\right) = \frac{1}{2} \sqrt{y}\sqrt{1-y}^{n-3} \Vol V_2(\R^{n-1}).
	\]
	Since the fibers of the Stiefel projection $V_2(\R^n) \to G_2(\R^n)$ are copies of $O(2)$, we just multiply the above by $\Vol O(2) = 4\pi$ to get the volume of $F_i^{-1}(y)$ as a hypersurface in $V_2(\R^n)$, yielding the expression in \eqref{eq:propVolumePlanar}.	% 
\end{proof}

We can now compute the moments of the edgelength distribution and so recover the results of Section~\ref{sec:moments}: the $p$th moment of edgelength for spatial arms is computed in this style as
\[
	\int_0^2 y^p \phi_3^\mathrm{Arm}(y)\, dy = \int_0^2 \frac{(2n-2)(2n-1)}{2^{n-1}} y^{p+1} (2-y)^{2n-3} \, dy = 2^p\frac{\Beta(p,2n)}{\Beta(p,2)},
\]
agreeing with the results of Proposition~\ref{prop:moments}. This probability distribution function is easy to confirm by experiment, as shown in Figure~\ref{fig:pdfcheck}.

\begin{figure}[ht]
\begin{overpic}[width=2.5in]{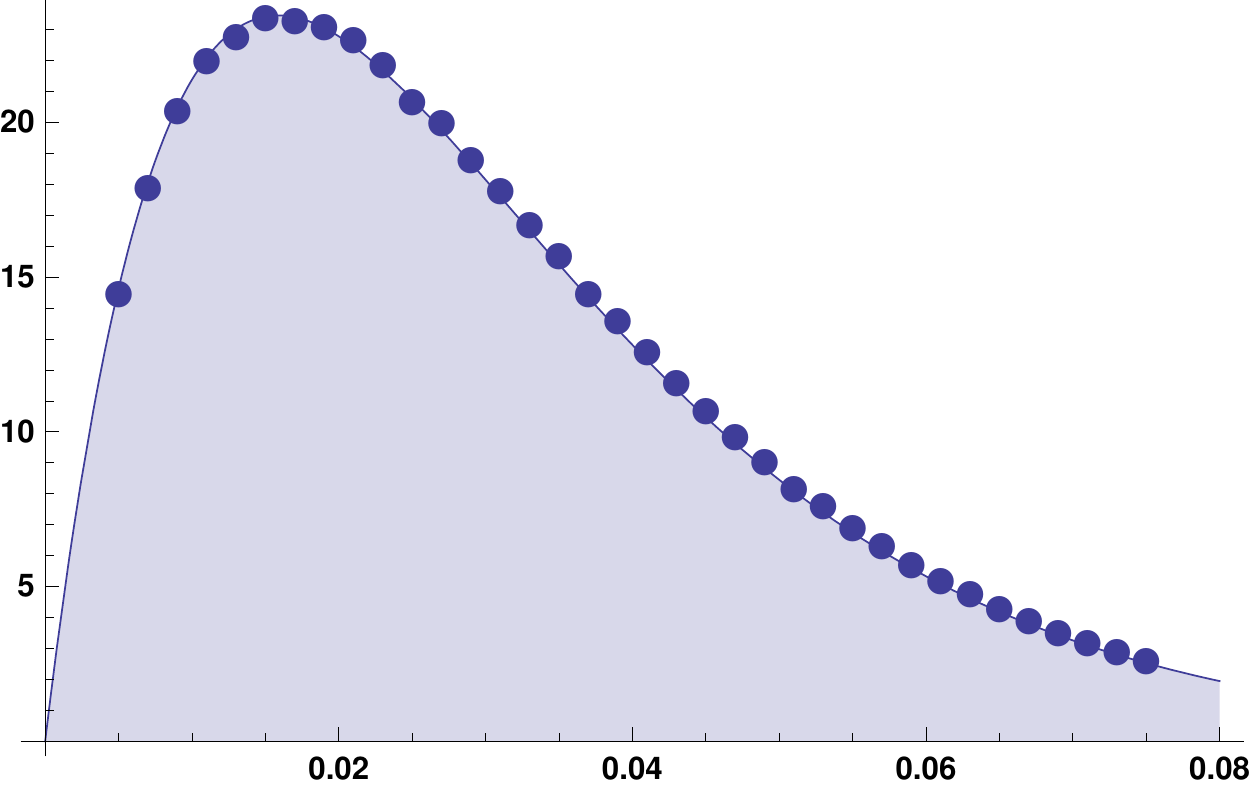}
\put(-10,64){$P$}
\put(101,2){$y$}
\end{overpic} 
\caption{This plot shows squared edgelength data from an ensemble of $10^5$ 64-edge arms generated by Erica Uehara. The smooth curve shows the function $\phi_3^\mathrm{Arm}(y)$ for $n=64$, while the data points are normalized bin counts from a histogram of the squared edge lengths of the polygons in our sample. An interesting observation is that this data also fits very well to the much simpler function $4096 y e^{-64 y}$, which is only $\simeq 0.03343$ from $\phi_3^\mathrm{Arm}(y)$ in $L^2$ distance over the interval $y \in [0,2]$. This fitting function is quite simple and hence should be useful for actual theoretical
analysis.
\label{fig:pdfcheck}
}
\end{figure}

More interestingly, we can make some precise statements about how close to equilateral uniformly sampled arms and polygons are. For example, the following proposition implies that all edgelengths are asymptotically of the same order as the expected value $\nicefrac{2}{n}$.

\begin{proposition}\label{prop:edgelengthBound}
	The probability that a randomly sampled $n$-edge spatial arm or polygon has $i$th edgelength greater than $c\frac{\log n}{n}$ for $c > 0$ satisfies
	\[
		P\left[|e_i| \geq c \frac{\log n}{n}\right] < \frac{1+c \log n}{n^c}
	\]
	provided that $n$ is large enough that $c > \frac{4}{\log n}$.
	
	Similarly, the probability that a randomly sampled $n$-edge planar arm or polygon has $i$th edgelength greater than $c \frac{\log n}{n}$ for $c > 0$ satisfies
	\[
		P\left[|e_i| \geq c \frac{\log n}{n}\right] < \frac{1}{n^{\nicefrac{c}{2}}}
	\]
	again provided that $c > \frac{4}{\log n}$.
\end{proposition}

\begin{remark}
	Since $e^4 \approx 54.6$, we have
	\[
		\frac{4}{\log n} < 1 \text{ for } n \geq 55
	\]
	and so Proposition~\ref{prop:edgelengthBound} holds with $c=1$ for $n \geq 55$.
\end{remark}

% \begin{proposition}\label{prop:edgelengthBound}
% 	The probability that a randomly sampled $n$-edge spatial arm or polygon has $i$th edgelength greater than $c\frac{\log n}{n}$ for $c > 0$ satisfies
% 	\[
% 		P\left[|e_i| \geq c \frac{\log n}{n} \right] < \frac{1+c \log n}{n^{c- \alpha}}
% 	\]
% 	for any $\alpha$ such that $\frac{n-2}{\log n} \log\left(1 - \frac{c \log n}{n} \right) + c<\alpha<c$. 
% 	
% 	For $n$-edge planar arms and polygons the corresponding bound is
% 	\[
% 		P\left[|e_i| \geq c \frac{\log n}{n}\right] < \frac{1}{n^{\nicefrac{(c-2\alpha)}{2}}}
% 	\]
% 	for any $\alpha$ such that $\frac{\nicefrac{n}{2}-1}{\log n} \log\left(1 - \frac{c\log n}{n}\right) + \frac{c}{2} < \alpha < \frac{c}{2}$. 
% 	
% 	For example, in all cases if $c>\nicefrac{1}{10}$, then the above bounds hold with $\alpha = \nicefrac{1}{10}$ for any $n \geq 10$.	% 
% 		% for any $0< \alpha < c$ provided $n$ is sufficiently large. In particular, if $c > \nicefrac{1}{10}$, then the above bound holds with $\alpha = \nicefrac{1}{10}$ for any $n \geq 4$.
% \end{proposition}

Both $\frac{1+c \log n}{n^c}$ and $\frac{1}{n^{\nicefrac{c}{2}}}$ go to zero as $n \to \infty$; therefore the space of $n$-edge arms and polygons with $i$th edgelength greater than $c \frac{\log n}{n}$ has measure approaching zero for large $n$.

\begin{proof}[Proof of Proposition~\ref{prop:edgelengthBound}]
	We prove the proposition for spatial arms and planar polygons; the other two cases are completely parallel.
	
	For spatial arms, the probability that the $i$th edgelength is greater than $c \frac{\log n}{n}$ is given by
	\begin{equation}\label{eq:armLargeEdges}
		\int_{c\frac{\log n}{n}}^2 \phi_\mathrm{Arm}^3(y) \, dy = \left(1-c\frac{\log n}{2n}\right)^{2n-2} \left(1+ c \log n - c \frac{\log n}{n}\right).
	\end{equation}
	The second term on the right hand side is less than $1 + c \log n$, taking care of the numerator in the bound. As for the first term, if $c>\frac{4}{\log n}$, then
	\[
		\left(1 - c\frac{\log n}{2n}\right)^{2n-2} < \left(1 - \frac{4}{2n}\right)^{2n-2}.
	\] 
	Since this expression is monotone increasing in $n$ and limits to $e^{-4}$ as $n \to \infty$, it must be less than
	\[
		e^{-4} = \left(n^{\nicefrac{1}{\log n}}\right)^{-4} = n^{-\nicefrac{4}{\log n}} < n^{-c},
	\]
	again using the fact that $c > \frac{4}{\log n}$. Therefore, $\left(1-c\frac{\log n}{2n}\right)^{2n-2} < \frac{1}{n^c}$ and the result follows.
	
	Notice that the expression on the right hand side of \eqref{eq:armLargeEdges} behaves asymptotically like the bound $\frac{1+c \log n}{n^c}$, so this bound is optimal whenever it holds. 
	
	For planar polygons, the probability that the $i$th edgelength is greater than $c \frac{\log n}{n}$ is given by
	\[
		\int_{c\frac{\log n}{n}}^1 \phi_\mathrm{Pol}^2(y) \, dy = \left(1 - c \frac{\log n}{n}\right)^{\nicefrac{n}{2}-1}.
	\]
	Then $c > \frac{4}{\log n}$ ensures that
	\[
		\left(1 - c \frac{\log n}{n}\right)^{\nicefrac{n}{2}-1} < \left(1 - \frac{4}{n}\right)^{\nicefrac{n}{2}-1},
	\]
	which in turn is monotone increasing in $n$ and so is less than its limiting value of 
	\[
		e^{-2} = \left(n^{\nicefrac{1}{\log n}}\right)^{-2} = n^{-\nicefrac{2}{\log n}} < n^{-\nicefrac{c}{2}},
	\]
	as desired. Again, the actual value and the bound are asymptotically the same, so this bound is optimal.
\end{proof}

Proposition~\ref{prop:edgelengthBound} implies that, for large $n$, we should expect all edgelengths of $n$-edge arms and polygons to be of the same order as the expected value $\nicefrac{2}{n}$. However, the space of arms or polygons with $i$th edgelength greater than any fixed multiple of the expected value has positive measure. For example, the proportion of space $n$-gons with $i$th edgelength greater than $c \frac{2}{n}$ for $c > 0$ is
\[
	P\left[|e_i| \geq c \frac{2}{n} \right] = \int_{c \frac{2}{n}}^1 \phi_\mathrm{Pol}^3(y) \, dy = \left(1-  \frac{2c}{n}\right)^{n-2} \left(1 + 2c -  \frac{4c}{n}\right)
\]
(provided, of course, that $c \frac{2}{n} < 1$). This quantity is monotone decreasing in $n$ and hence is bounded below by its limiting value $\frac{1+2c}{e^{2c}}$. Therefore, 
\[
	P\left[|e_i| \geq c \frac{2}{n} \right] > \frac{1+2c}{e^{2c}}
\]
for all $n$. Similar results hold for the other arm and polygon spaces.

\clearpage

\section{Asymptotic Comparison of Polygons and Arms}
\label{sec:asymptotics}
It is a natural intuition that ``sufficiently short'' sections of a closed polygon should behave like corresponding sections of polygonal arms. Since the closure constraint is global, its local effect should vanish as $n \rightarrow \infty$. Expressing distance along the curve in the fraction $\delta = k/n$, it is easy to see that we have
\begin{equation*}
(1-\delta) \frac{n}{n+1} < \frac{E(\chord(\delta),\Pol_i(n))}{E(\chord(\delta),\Arm_i(n))} < (1-\delta) \frac{n}{n-1}.
\end{equation*}
We can see from this formula that the fractional distance along the curve (in vertices) is what matters: for any fixed $k$, $\delta \rightarrow 0$ as $n\rightarrow \infty$ and the expected values of $\chord$ converge rather quickly for arms and polygons. An even simpler formula holds for equilateral polygons, as 
\begin{equation*}
\frac{E(\chord(\delta),\ePol_i(n))}{E(\chord(\delta),\eArm_i(n))} = (1-\delta) \frac{n}{n-1}.
\end{equation*}

The radius of gyration formulae in Propositions~\ref{prop:egyradius} and~\ref{prop:eqgyradius} are similarly interesting: we can see that, asymptotically, the (squared) radius of gyration of an arm is expected to be exactly twice as large as the corresponding (squared) radius for a closed polygon.

\section{Sampling in Arm Space and Polygon Space}
\label{sec:sampling}

Sampling closed polygons is traditionally fairly difficult. Orlandini et al.\ \cite{Orlandini:2007kn} provide an overview of the standard methods, many of which depend on establishing Markov chains which are ergodic on equilateral closed polygon space and then iterating the chain until the resulting distribution on polygon space is close to uniform. Grosberg and Moore~\cite{Moore:2005fh} discuss some potential difficulties with these iterative methods, and give a method for explicitly sampling random equilateral closed polygons for small numbers of edges by computing the conditional probability distribution of the $n+1$-st edge based on the choice of the first $n$ edges  (see~\cite{Diao:2011ie},\cite{2011JPhA...44R9501D}, and \cite{Diao:wt} for conditional probability methods applied to the even more difficult problem of sampling equilateral closed polygons confined to a sphere, and  \cite{Varela:2009cd} for an alternate approach to generating ensembles of equilateral closed polygons). These conditional probability algorithms are somewhat challenging to implement and fairly slow, requiring high precision arithmetic and scaling with $O(n^3)$. 

If one is willing to shift one's focus from equilateral polygons to our polygon spaces where the edgelengths obey beta distributions, there is a substantial computational reward: direct sampling in our spaces of closed length 2 (but \emph{not} equilateral) polygons is fast and easy, taking only a few lines of code and scaling with $O(n)$. In this section, we describe our sampling algorithm and provide some numerical results which confirm the theoretical results above.

Sampling a polygon in arm space is equivalent to choosing points uniformly on the sphere $S^{4n-1}$ or $S^{2n-1}$. There is a well-established literature for this problem, but we mention that it suffices to generate a vector of independent standard Gaussians and then normalize it. Sampling in our closed polygon spaces requires us to sample with respect to Haar measure on the Stiefel manifold $V_{2}(\C^n)$ or $V_2(\R^n)$. Chikuse~\cite{Chikuse:2003we} gives several algorithms for this. The simplest is 
\begin{proposition}
If $u$ and $v$ are generated uniformly on $S^n$, the Gram-Schmidt orthonormalization procedure applied to $(u,v)$ yields an orthonormal frame $(u',v')$ which is uniformly distributed on the Stiefel manifold $V_2(\R^n)$, or, if Gram-Schmidt is performed with the Hermitian inner product for vectors on the complex unit sphere, a frame uniformly distributed on $V_2(\C^n)$.
\label{prop:gs_sample}
\end{proposition}
For the convenience of the reader, here is explicit pseudocode for this method. We assume the existence of a function $\proc{Gaussian}$ which gives a random value sampled from a standard Gaussian distribution, such as the GNU Scientific Library function \verb|gsl_ran_ugaussian()|. The entries in the arrays $A$, $B$, $\id{FrameA}$, and $\id{FrameB}$ are assumed to be complex, the $\proc{Conj}$ function is assumed to give the complex conjugate $x - \I y$ of the complex number $x + \I y$, and the $\proc{Re}$ and $\proc{Im}$ functions are assumed to give the real and imaginary parts of a complex number. To generate random planar polygons, delete the expression $I*\proc{Gaussian()}$ both places it occurs in the first loop of $\proc{Random-Space-Polygon}$.

\begin{codebox}
\Procname{$\proc{CplxDot}(V,W)$}
\zi \Comment Compute the Hermitian dot product of two complex $n$-vectors. 
\zi \For $\id{ind} = 1$ \To $n$ 
\zi     \Do
       $\id{Dot} += V[\id{ind}] * \proc{Conj}(W[\id{ind}])$
 \End
\zi  \Return \id{Dot}
\End
\end{codebox}

\begin{codebox}
\Procname{$\proc{Normalize}(V)$}
\zi \Comment Normalize a complex $n$-vector to unit length.
\zi \For $\id{ind} = 1$ \To $n$ 
\zi     \Do
       $\id{UnitV}[\id{ind}] = V[\id{ind}]/\proc{Sqrt}(\proc{CplxDot}(V,V))$ 
    \End
\zi  \Return \id{UnitV}  
\End
\end{codebox}

\begin{codebox}
\Procname{$\proc{HopfMap}(a,b)$}
\zi \Comment Compute the vector in $\R^3$ given by the 
\zi \Comment Hopf map applied to the quaternion $a + b \J$.
\zi \Return ( a * \proc{Conj}(a) - b * \proc{Conj}(b), 2 \proc{Re}(a * \proc{Conj}(b)), 2 \proc{Im}(a * \proc{Conj}(b)) )
\End
\end{codebox}

\begin{codebox}
\Procname{$\proc{Random-Space-Polygon}(n)$}
\zi \Comment Produce edge vectors $\id{Edge}[\id{ind}]$ for a 
\zi \Comment random closed space polygon of length 2.

\zi \Comment 1. Generate a frame with Gaussian coordinates.

\zi \For $\id{ind} = 1$ \To $n$
\zi   \Do
        $A[\id{ind}] = \proc{Gaussian}() + I * \proc{Gaussian}()$
\zi     $B[\id{ind}] = \proc{Gaussian}() + I * \proc{Gaussian}()$
\zi   \End

\zi \Comment 2. Perform Gram-Schmidt to get $\id{FrameA}$ and $\id{FrameB}$.
\zi \For $\id{ind} = 1$ \To $n$
\zi   \Do
        $\id{FrameA}[\id{ind}] = A[\id{ind}]$
\zi     $\id{FrameB}[\id{ind}] = B[\id{ind}] - (\proc{CplxDot}(B,A)/\proc{CplxDot}(A,A)) A[\id{ind}]$
      \End
\zi   $\id{FrameA} = \proc{Normalize}(\id{FrameA})$, $\id{FrameB} = \proc{Normalize}(\id{FrameB})$
\zi

\zi \Comment 3. Apply the Hopf map coordinate-by-coordinate.

\zi \For $\id{ind} = 1$ \To $n$
\zi   \Do
	    $\id{Edge}[\id{ind}] = \proc{HopfMap}(\id{FrameA}[\id{ind}],\id{FrameB}[\id{ind}])$.
\zi      \End

\zi  \Return \id{Edge}
\End

\end{codebox}

Since this is a linear-time algorithm, this is a very fast way to sample random polygons uniformly with respect to our measure. A reference implementation in C of this method is provided as part of Cantarella's \texttt{plCurve} library, which is open-source and freely available. Figure~\ref{fig:examples} shows some space polygons generated by this library. 
\begin{figure}
\begin{tabular}{ccc}
\includegraphics[height=2in]{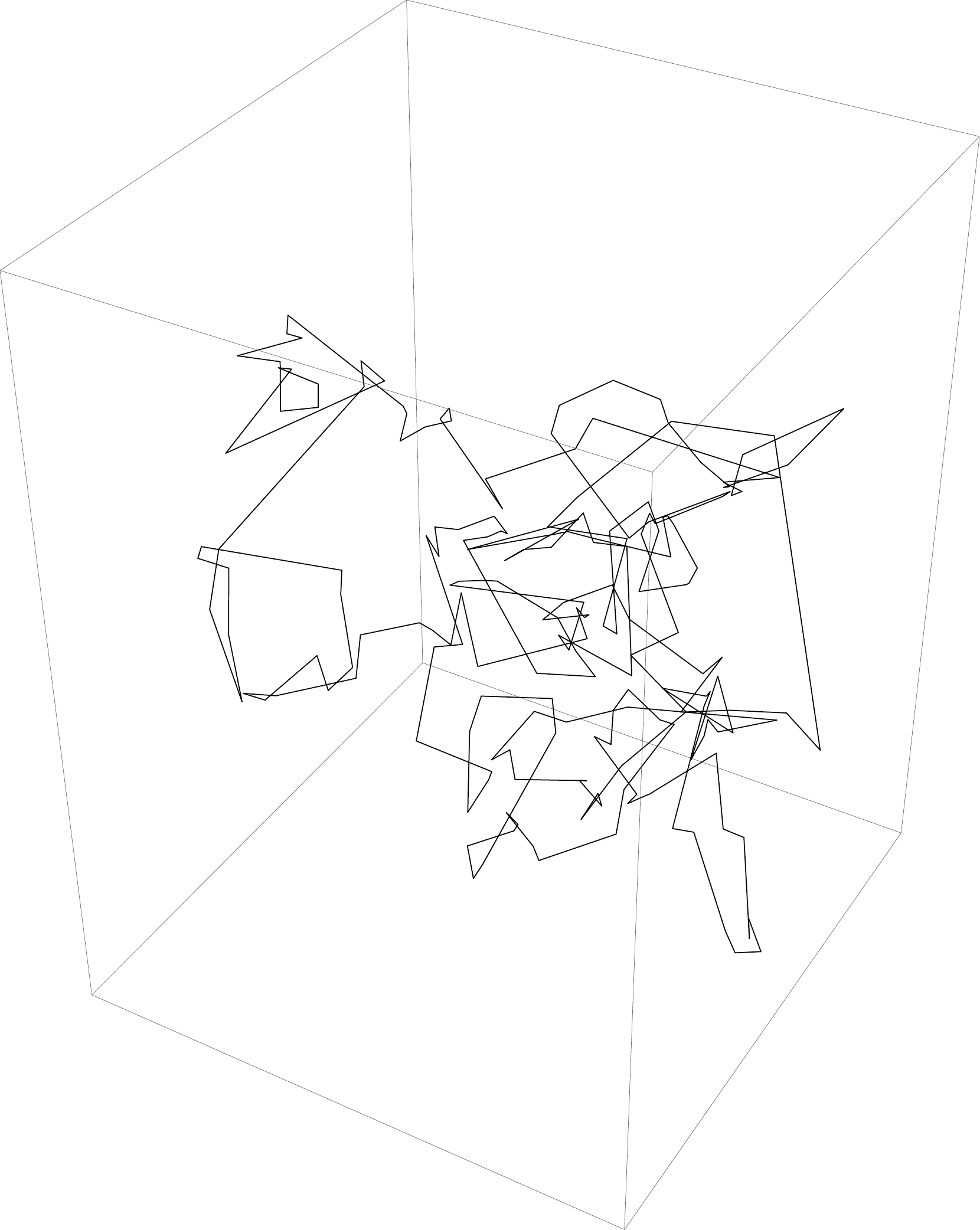} & 
\includegraphics[height=2in]{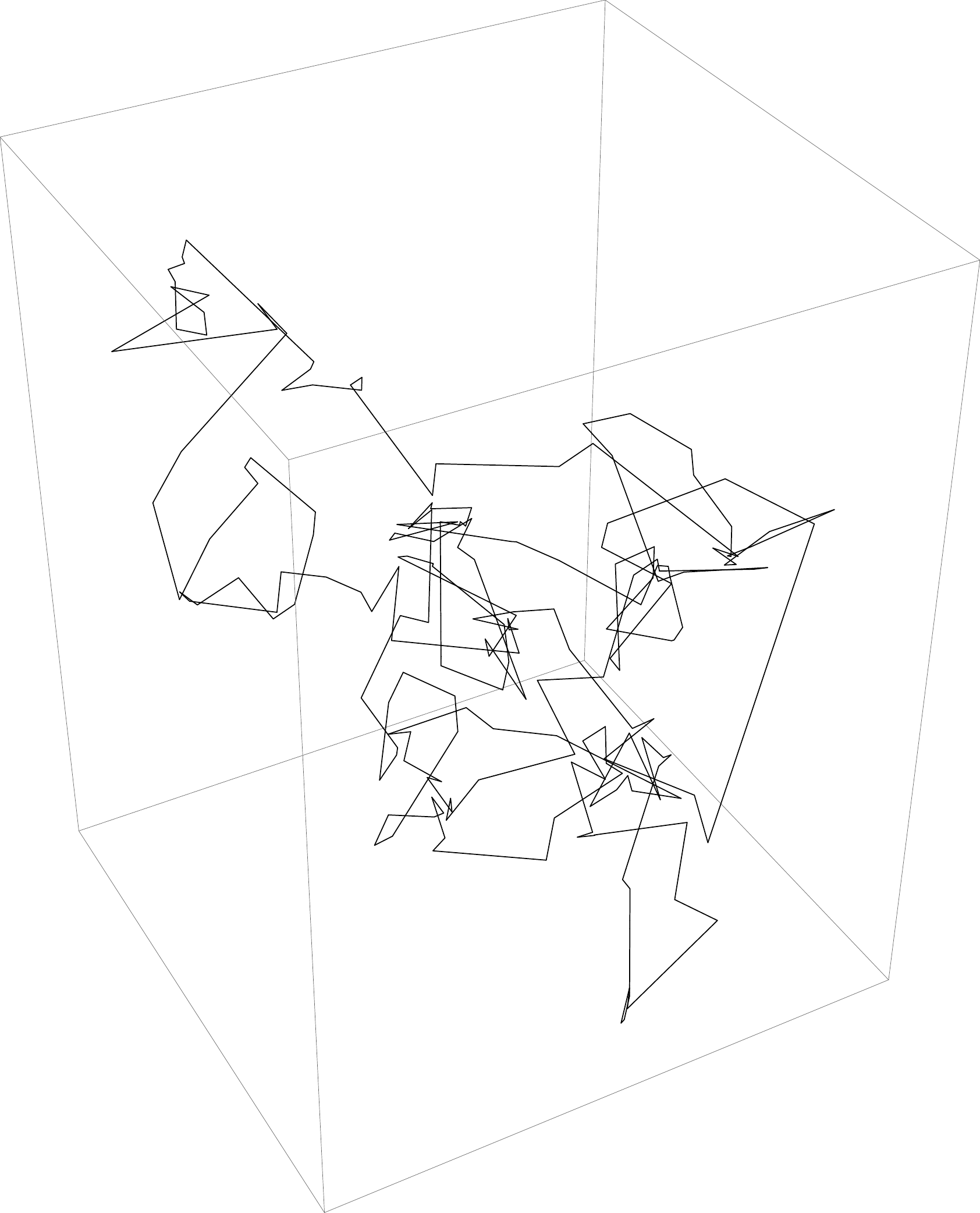} & 
\includegraphics[height=2in]{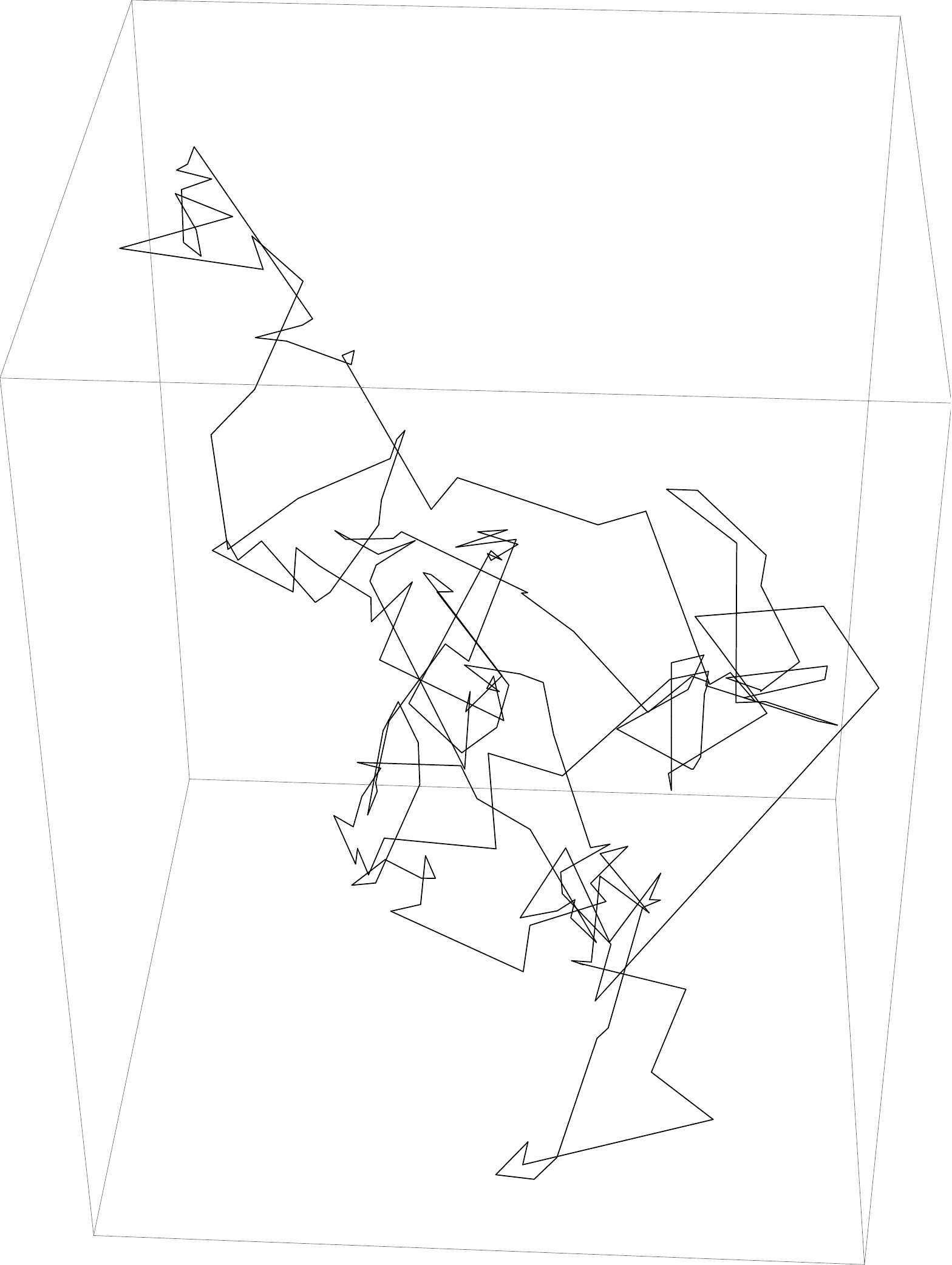} \\
200 & 200 & 200 \\
\includegraphics[height=2in]{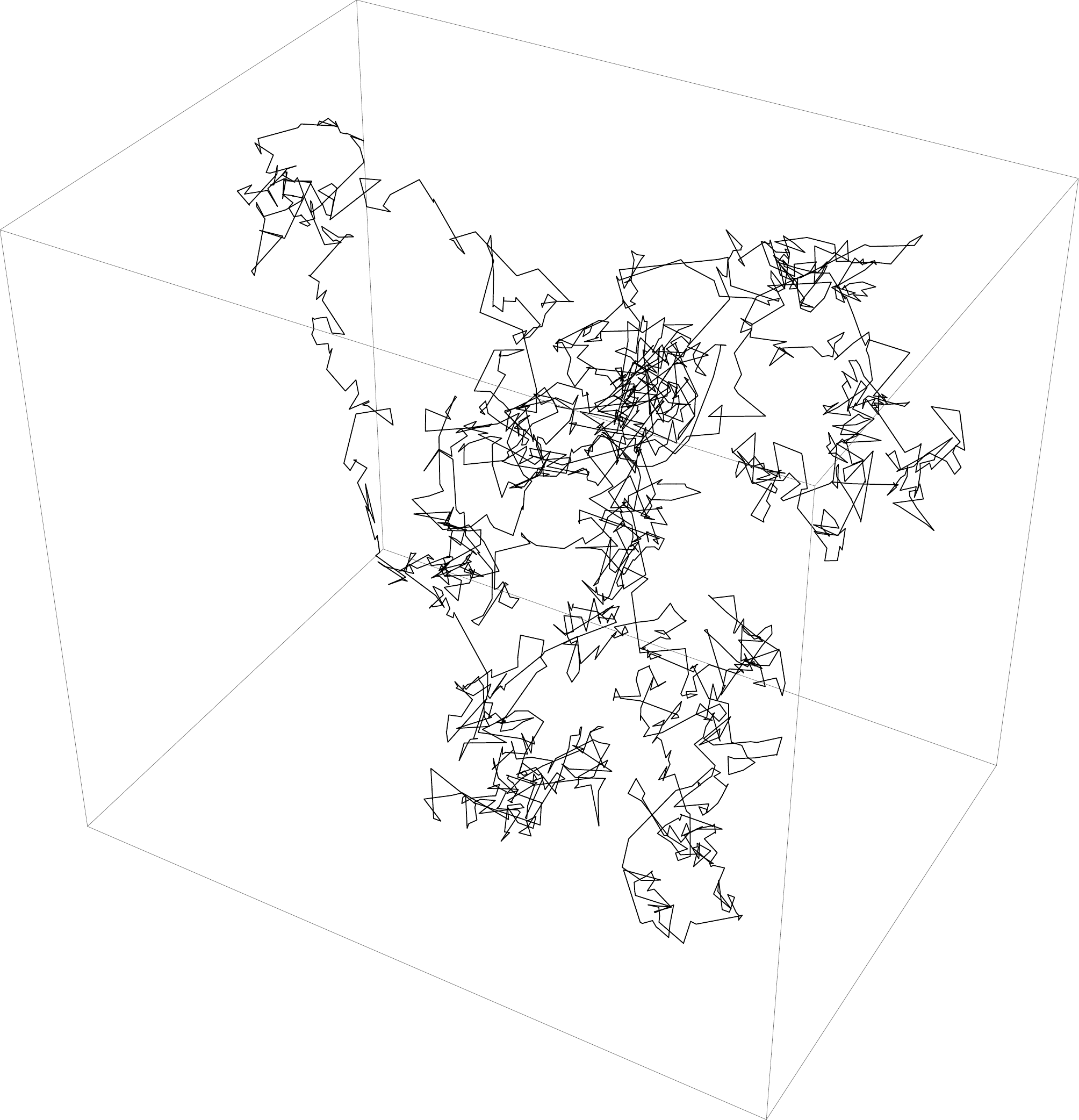} & 
\includegraphics[height=2in]{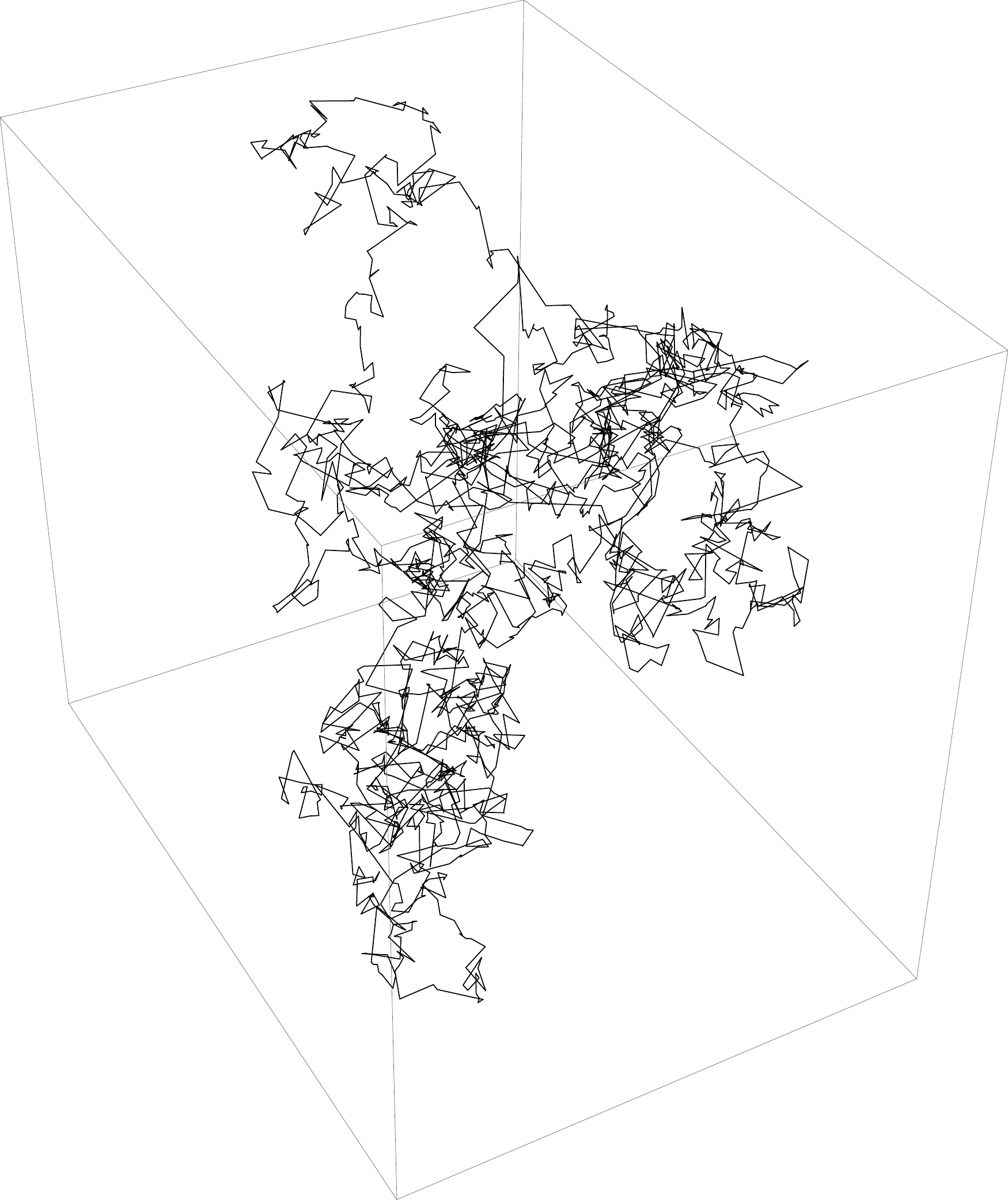} & 
\includegraphics[height=2in]{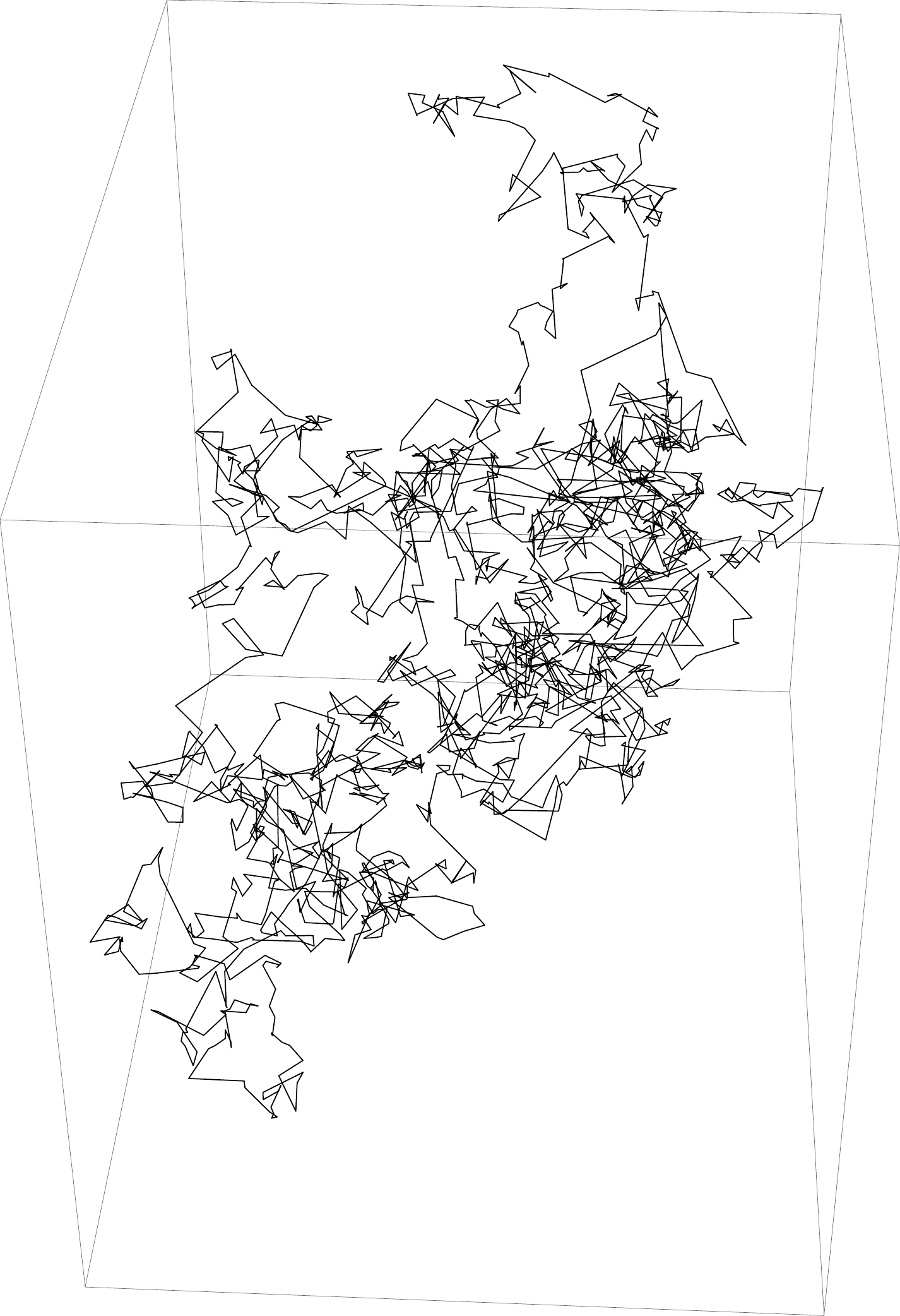} \\
2,000 & 2,000 & 2,000 \\
\includegraphics[height=2in]{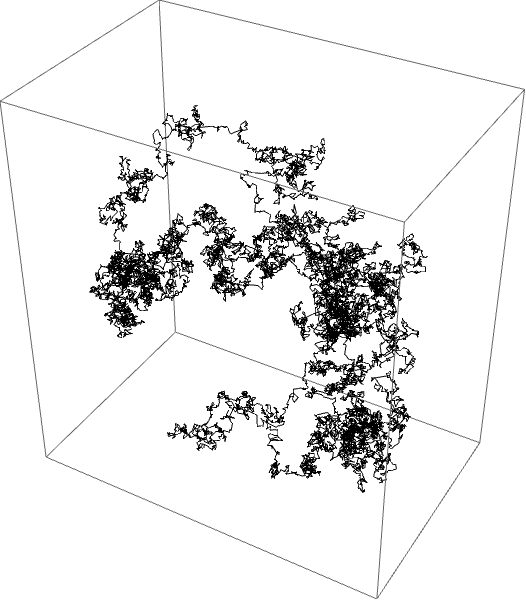} & 
\includegraphics[height=2in]{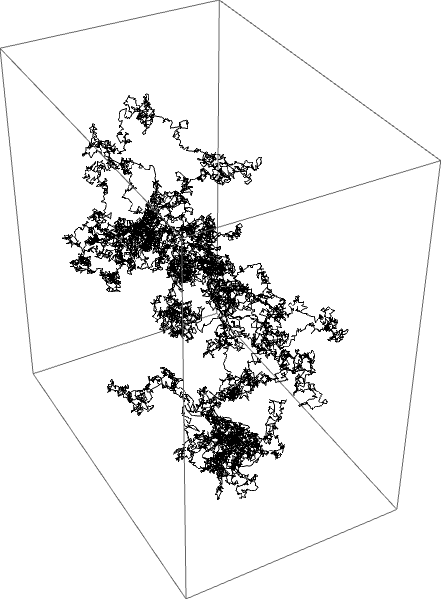} & 
\includegraphics[height=2in]{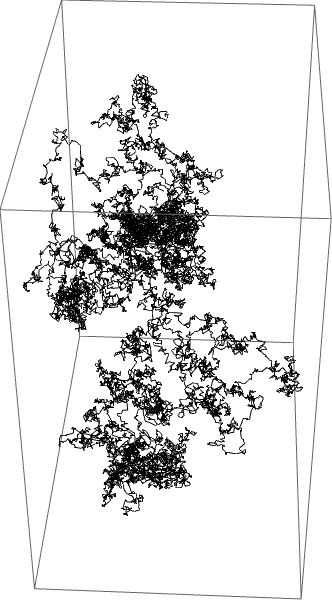} \\
20,000 & 20,000 & 20,000 
\end{tabular}
\caption{This figure shows three views of three space polygons directly sampled from our distribution on $\Pol_3(n)$ by sampling frames in the Stiefel manifold. The polygons have 200, 2,000, and 20,000 edges, respectively.
\label{fig:examples}}
\end{figure}
To provide a check on our theory and our code, we can now test the predictions of Propositions~\ref{prop:expected chords} and~\ref{prop:egyradius} against the data generated by the library. Figure~\ref{fig:msc} shows the comparison between theory and experiment for mean squared chordlength and Figure~\ref{fig:gyradius} shows the corresponding comparison for radius of gyration. As expected, we see that it is easy to verify our theorems to several digits numerically.
\begin{figure}
\begin{tabular}{c@{\hspace{0.5in}}c}
\begin{overpic}[width=2.5in]{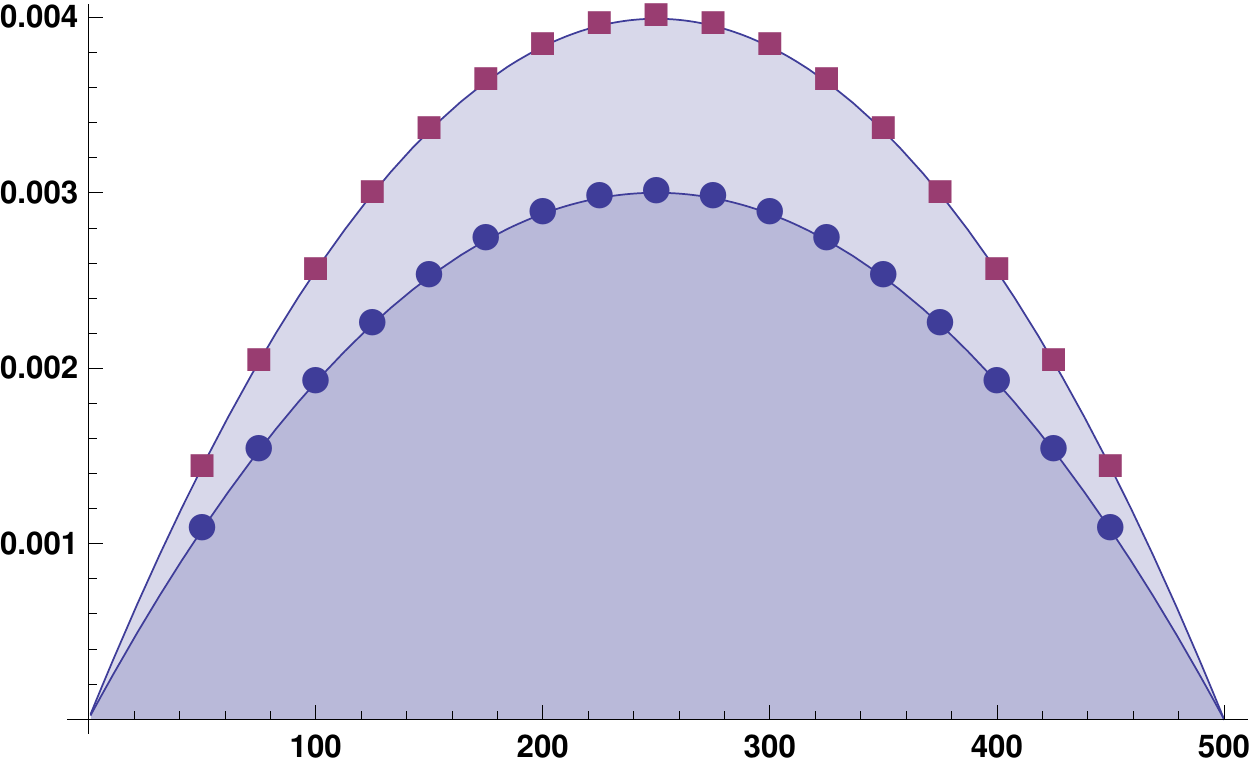}
\put(-10,64){$E(\chord(k))$}
\put(101,2){$k$}
\end{overpic} 
&
\begin{overpic}[width=2.5in]{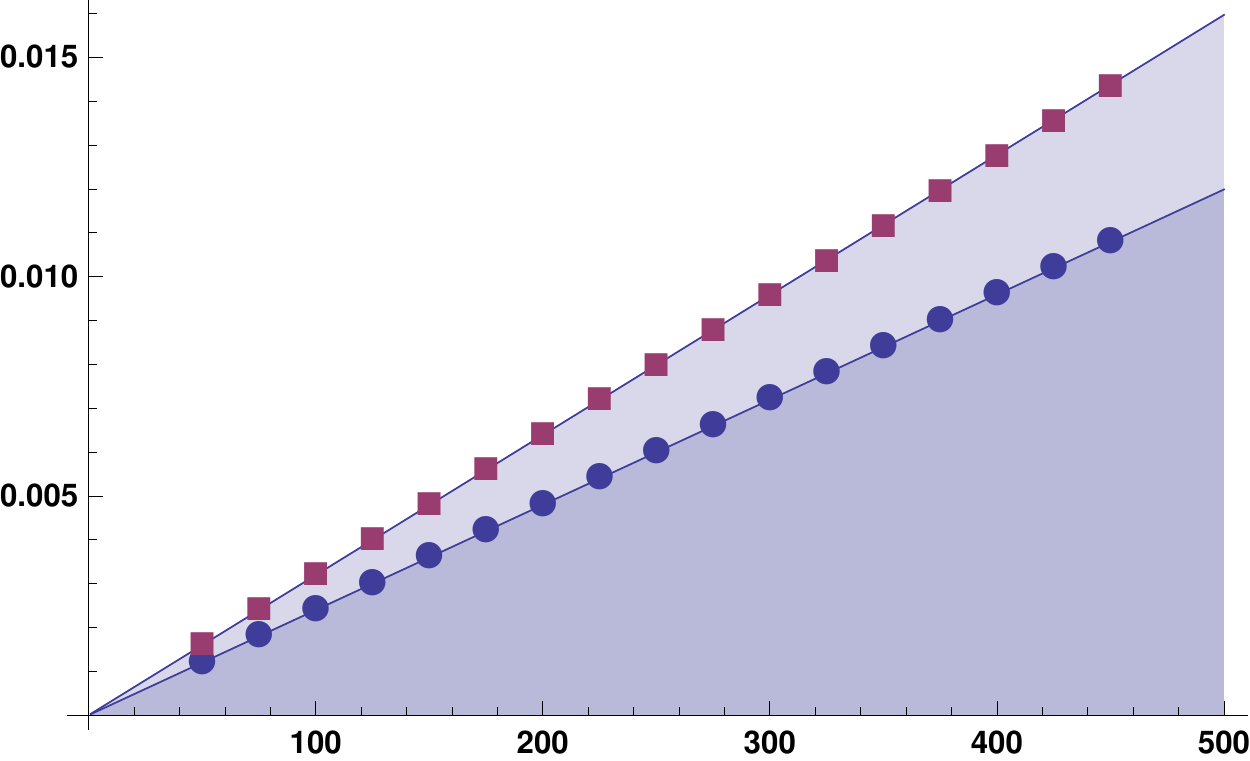}
\put(-10,64){$E(\chord(k))$}
\put(101,2){$k$}
\end{overpic}
\\
$\Pol_2(500)$ and $\Pol_3(500)$
&
$\Arm_2(500)$ and $\Arm_3(500)$
\end{tabular}
\caption{These graphs show the comparison between the mean squared chordlength predictions of Proposition~\ref{prop:expected chords} and data generated by the Gram-Schmidt sampling algorithm of Section~\ref{sec:sampling}. We generated a sample of 200,000 random 500-gons and computed the mean of squared chordlength for all chords spanning $k$ edges in each polygon. The mean of those means is plotted above. Since this is a mean over $\simeq 10^9$ values, it is not surprising that it agrees extremely well with the predictions of Proposition~\ref{prop:expected chords}. To give a specific example, the computed mean for the squared length of chords skipping 250 edges in a closed space 500-gon was $0.00300383$, while the predicted value was $\simeq 0.003000012000$. These values differ by $\simeq 0.127 \%$. The left-hand graph shows theory and experiment for closed plane (upper curve) and space (lower curve) polygons, while the right-hand graph shows the same for open plane (upper curve) and space (lower curve) polygons. Each experiment takes about 18 seconds to run on a desktop Mac.
\label{fig:msc}}
\end{figure} 

\begin{figure}
\begin{overpic}[width=3.5in]{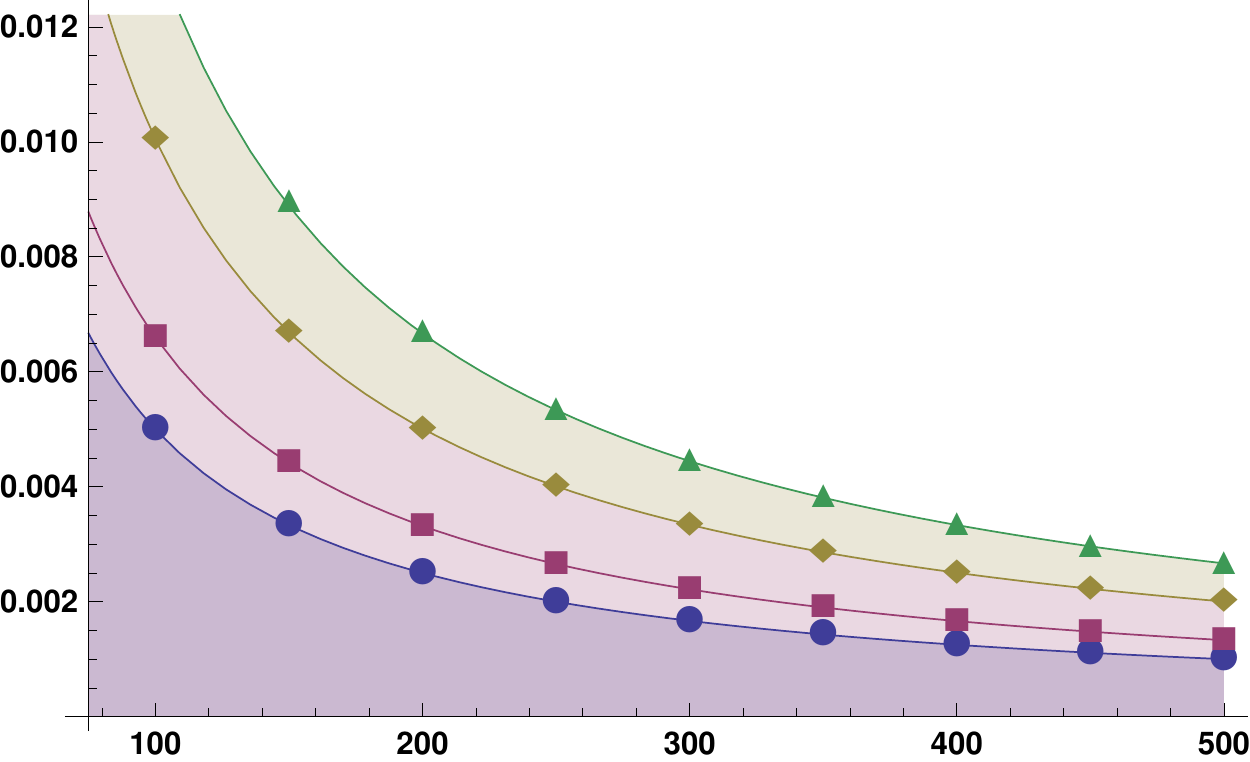}
\put(-7,64){$E(\gyradius(n))$}
\put(101,2.5){$n$}
\end{overpic} 
\\
$\Pol_3(n)$, $\Pol_2(n)$, $\Arm_3(n)$, and $\Arm_2(n)$
\caption{This graph shows the comparison between the radius of gyration predictions of Proposition~\ref{prop:egyradius} and data generated by the Gram-Schmidt sampling algorithm of Section~\ref{sec:sampling}. We generated a sample of 40,000 random polygons with 100, 200, 300, 400, and 500 edges in each of our four classes of polygons. The mean gyradius of each of these samples is plotted above, along with the curves from Proposition~\ref{prop:egyradius}. In order from bottom to top, the four curves are for $\Pol_3(n)$ (circles), $\Pol_2(n)$ (squares), $\Arm_3(n)$ (diamonds), and $\Arm_2(n)$ (triangles). Since computing the radius of gyration of an $n$-gon is an $O(n^2)$ computation, the time to compute gyradius for the polygons in the ensemble dominated the time required to generate the ensemble of polygons. The resulting total experiment takes about 3 minutes to run for each class of polygons. As before, the convergence to our predictions is very good: for closed, space 300-gons, our prediction for the expected radius of gyradius is $1/600 \simeq 0.001666667$. The mean of gyradius for our sample was $0.00166576$, which differs by $\simeq 0.054 \%$.\label{fig:gyradius}}
\end{figure}

The sampling algorithm above does not directly generate an ensemble of equilateral polygons: we are sampling only in the entire polygon spaces $\Pol_i(n)$ and $\Arm_i(n)$ and not in the codimension $n-1$ subspaces $\ePol_i(n)$ or $\eArm_i(n)$. However, we can generate polygons uniformly sampled from a neighborhood of $\ePol_i(n)$ or $\eArm_i(n)$ by rejection sampling: we generate a larger ensemble of polygons using the above algorithm and then throw out the polygons with edges longer than a certain bound. A reference implementation of this polygon generator is provided in \texttt{plCurve}. To estimate the performance of the method, we can use the pdf of Proposition~\ref{prop:edgelengthBound}. If we make the simplifying assumption that the edgelengths are independent (of course, this is not literally true, but it may be asymptotically true for large numbers of edges) we can estimate the probability of success for the rejection sampler for a given upper bound on edgelengths. Figure~\ref{fig:rejectionsampling} shows a set of such computations carried out explicitly for $2000$-gons.

\begin{figure}
Probability $P$ that a random $2,\!000$-gon has maximum edgelength $< \lambda (\text{mean edgelength})$.\\
\vspace{0.25in}
 
\begin{tabular}{c@{\hspace{0.5in}}c}
\begin{overpic}[width=2.5in]{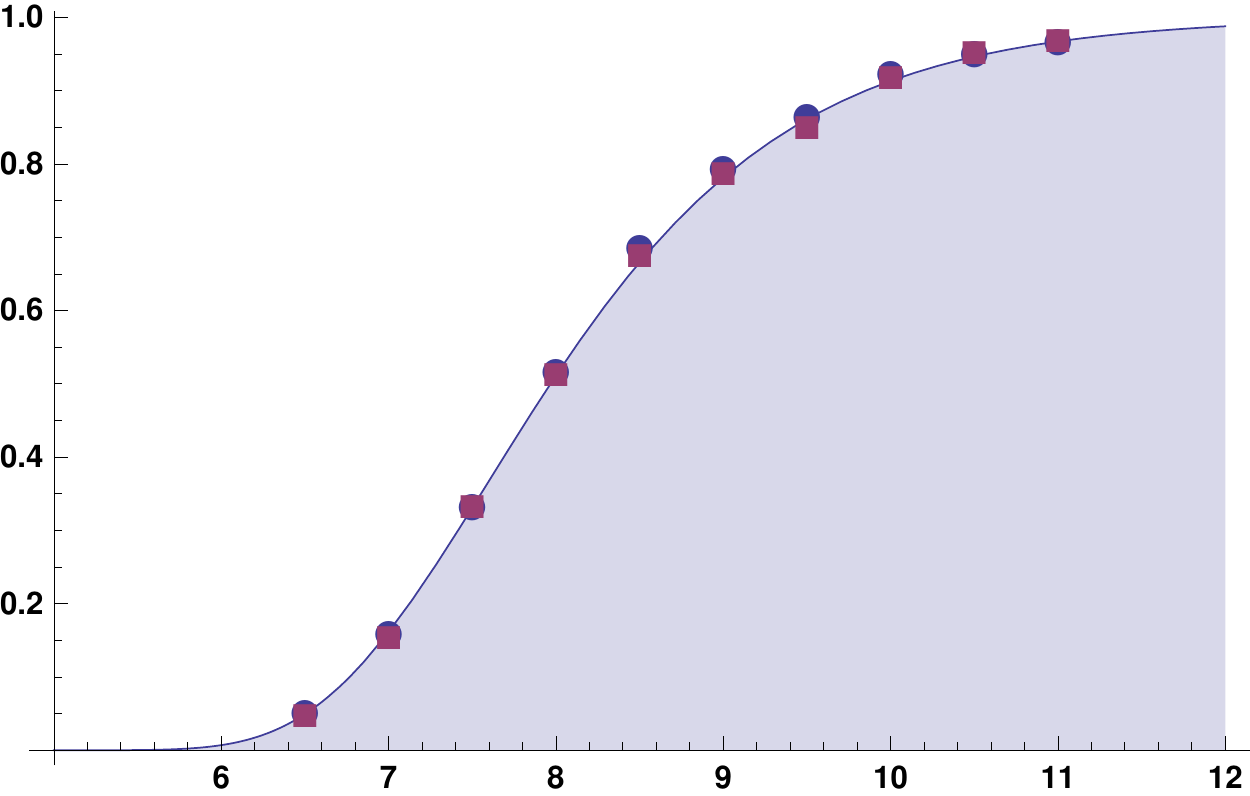}
\put(-1,66){$P$}
\put(101,2){$\lambda$}
\end{overpic} 
&
\begin{overpic}[width=2.5in]{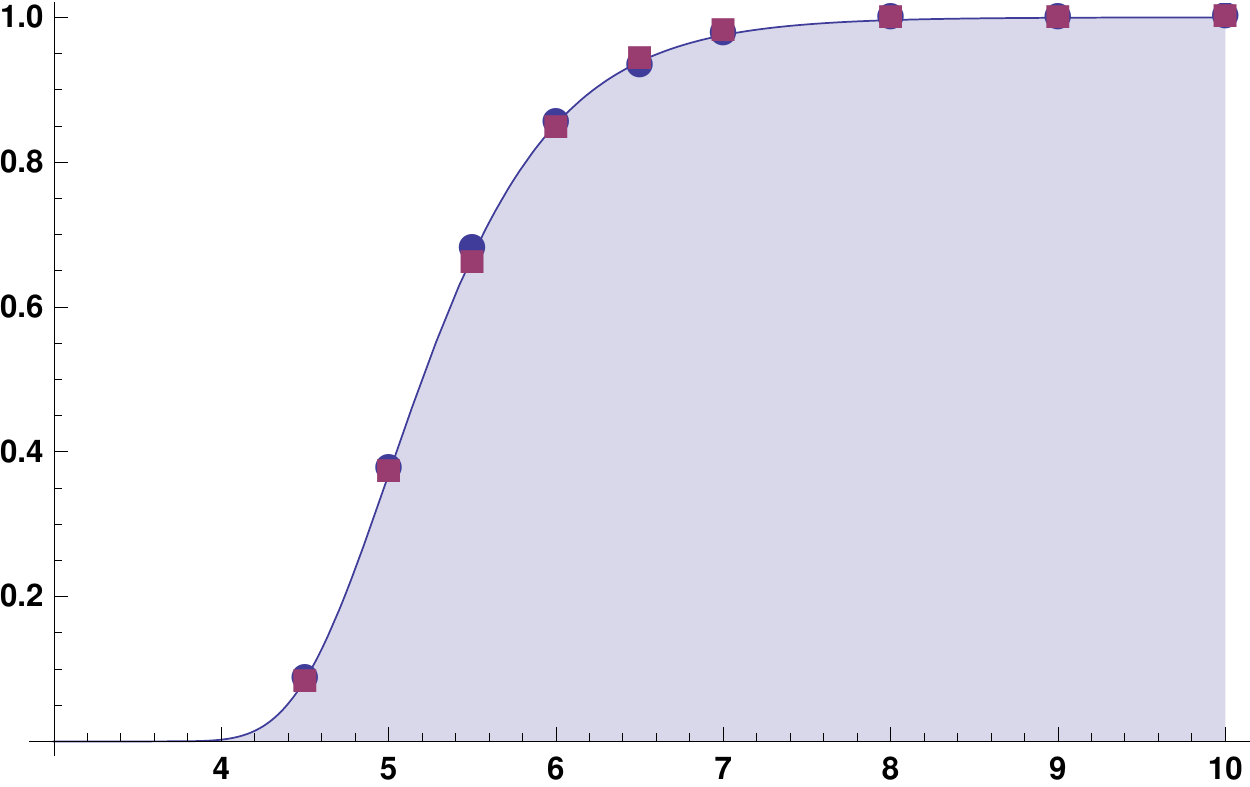}
\put(-1,66){$P$}
\put(101,2){$\lambda$}
\end{overpic}
\\
$\Pol_2(2,\!000)$ and $\Arm_2(2,\!000)$
&
$\Pol_3(2,\!000)$ and $\Arm_3(2,\!000)$
\end{tabular}
\caption{These graphs show data collected from rejection sampling applied to 2,000 edge polygons. For each of our four classes of polygons and various values of $\lambda$ we sampled until we had accepted 2,000 polygons with maximum edgelength less than $\lambda$ times the mean edgelength for 2,000-gons (1/1000). We recorded the fraction of accepted polygons for each $\lambda$. This data is plotted above; on the left for plane polygons (circles) and arms (squares) and on the right for space polygons (circles) and arms (squares). The curves show the comparison estimate computed from Proposition~\ref{prop:edgelengthBound} assuming that each edgelength is sampled independently from the distribution of Proposition~\ref{prop:pdfs}. As we can see, this estimate predicts the actual performance of the sampling algorithm quite well. For 2,000 edge plane polygons, it is quite efficient to generate polygons with $\lambda = 6.5$, as about $4.5\%$ of generated polygons are accepted, so our reference implementation generates a closed polygon in $\simeq 0.018$ seconds. For 2,000 edge space polygons, the variance of edgelength is lower, so it is easier to generate polygons with lower $\lambda$ values. With $\lambda = 4.5$, about $8\%$ of polygons are accepted, so our reference implementation generates a closed space 2,000-gon in $\simeq 0.013$ seconds.
\label{fig:rejectionsampling}}
\end{figure} 

We note that while the rejection sampler works as expected, it does not generate an ensemble of polygons which look much like a random sample of equilateral polygons from a statistical point of view. For instance, even when the rejection sampler accepts only 1 in 1,000 closed space 2,000-gons, it is still generating an ensemble of polygons where some members have a longest edge as much as 3.95 times the length of the mean edge. A numerical experiment shows that the second moment of the edgelength distribution of these polygons is $\simeq 1.46 \times 10^{-6}$, which is rather close to the value $6/(2,\!000*2,\!001) \simeq 1.499 \times 10^{-6}$ predicted for the entire ensemble of closed space 2,000-gons by Corollary~\ref{cor:moments} and rather far from the corresponding value of $4/2,\!000^2 \simeq 1.0 \times 10^{-6}$ for equilateral polygons. The mean chordlengths and radii of gyration for these polygons behave accordingly-- they are very close to the corresponding mean values over the entire polygon space and quite far from the mean values for equilateral polygons.

\section{Future Directions}

The sphere and Stiefel manifold techniques of this paper seem to open a large number of interesting opportunities for future exploration in polygon space. Of course, having obtained explicit formulae for the expected values of chord lengths, it is immediately desirable to start working out the higher moments of these distributions and indeed to express them explicitly as probability distributions. 

For example, one can certainly expect to say much more than the analysis of Section~\ref{sec:asymptotics} about how the pdfs governing ``short'' arcs of a closed polygon converge to corresponding pdfs for arcs of open polygonal arms. We believe this can be done with theorems of De Finetti type in probability such as \cite{Diaconis:2005p2630}. These give explicit bounds on how quickly the probability distributions of the individual coordinates of the frames in a Stiefel manifolds converge to independent (normalized) Gaussians. We have not yet investigated this question.

Given a space polygon, we may construct a plane polygon by projecting to a plane. It is natural to consider the relationship between the probability measure on plane polygons obtained by pushing forward our measure on space polygons using this construction and our original measure on plane polygons. They cannot be the same measure, of course: the projected polygons are shorter, with a variable total length of expected value $\pi/2$ (by Crofton's formula). In fact, even if we rescale the projected polygons to length $2$, the rescaled measure does not seem to be the same either: computing the average $\gyradius$ for 50,000 closed space 1024-gons projected to the plane and then rescaled to length 2 yields $5.27 \times 10^{-4}$ which is very far from the expected value of $\gyradius$ of $\simeq 6.504 \times 10^{-4}$ predicted by Proposition~\ref{prop:egyradius} for closed, length 2, plane 1024-gons.

Grosberg~\cite{Grosberg:2008dl} was able to use the expectation of the dot product of two edges in an equilateral closed random polygon (cf. Corollary~\ref{cor:edotproduct}) to argue that the expected value of total curvature for an $n$-segment closed (space) polygon is asymptotically $n \frac{\pi}{2} + \frac{3}{8}\pi$. Averaging total curvature over a sample of 50 million of our 5,000-gons gives an average total curvature of $ 7854.764997 \simeq 5000 \frac{\pi}{2} + 0.783364$. The ``surplus'' curvature is fairly close to $\frac{1}{4} \pi \simeq 0.785398$, which certainly suggests the conjecture that the corresponding expectation of total curvature for our polygons is $n \frac{\pi}{2} + \frac{1}{4}\pi$. We will address this conjecture in~\cite{FutureWork}.

The expected value of the dot product of two edges is fairly easy to compute (as in Grosberg's work or Corollaries~\ref{cor:dotproduct} and~\ref{cor:edotproduct}), but higher moments seem more challenging. More generally, a deeper understanding of the correlations between edges in either equilateral or non-equilateral polygons is desirable. In the case of equilateral polygons, the joint distribution of the dot products between all pairs of edges encodes this correlation completely, so a first step might be to compute either higher moments of the dot products or covariances of dot products. For non-equilateral polygon spaces the joint distribution of dot products remains interesting, but the correlations between edges are also dependent on the joint distribution of edgelengths. By Corollary~\ref{cor:covariance} the covariance of edgelengths is non-zero in our model, so as expected this joint distribution is non-trivial. 

There are also more detailed geometric structures available for investigation. Since our arm and polygon spaces are quotients of symmetric spaces by groups of isometries, we do not just have a measure on the arm and polygon spaces but a Riemannian metric, defined so that our projections are (almost everywhere) Riemannian submersions. This means that (for instance) optimal reconfigurations of closed or equilateral polygons can be obtained by following corresponding geodesics. More relevantly for the statistical physics community, the Laplace-Beltrami operator on the Stiefel manifolds is well-understood, which should allow us to make some rather precise statements about (intrinsic) Brownian motion of closed polygonal chains.

\section{Acknowledgements}
The authors are grateful to many friends and colleagues for helpful feedback on this paper and discussions of polygon space, including the anonymous referee, Malcolm Adams, Michael Berglund, Mark Dennis, Yuanan Diao, Claus Ernst, John Etnyre, David Gay, Alexander Grosberg, Danny Krashen, Rob Kusner, Matt Mastin, Ken Millett, Frank Morgan, Tom Needham, Jason Parsley, De Witt Sumners, Margaret Symington, Stu Whittington, Erica Uehara, Mike Usher, and Laura Zirbel.

\appendix
\section{The Proof of Proposition~\ref{prop:schord}} \label{appendix}

In this section we prove Proposition~\ref{prop:schord}, which we restate for convenience:

\begin{proposition}
	\[
		\schord(k,\mathcal{P}) = \frac{k(n-k)}{n(n-1)}\sum_{j=1}^n |e_j|^2 + 4 \frac{k(k-1)}{n(n-1)}\ell^2.% \left(\cos^2 2\theta + \sin^2 2\theta \langle \vec{a}, \vec{b}\rangle^2\right).
	\]
\end{proposition}

\begin{proof}
	From the definition of $\schord(k,\mathcal{P})$ and \eqref{eq:arm}, this is equivalent to proving
	\begin{multline}\label{eq:schordMain}
		\frac{4}{n!}\sum_{\sigma \in S_n} \left[\left(\cos^2\theta\sum_{j=1}^k a_{\sigma(j)}\bar{a}_{\sigma(j)} - \sin^2\theta\sum_{j=1}^k b_{\sigma(j)}\bar{b}_{\sigma(j)} \right)^2 + 4\sin^2\theta \cos^2\theta\left|\sum_{j=1}^k a_{\sigma(j)}\bar{b}_{\sigma(j)}\right|^2 \right] \\
		= \frac{k(n-k)}{n(n-1)}\sum_{j=1}^n |e_j|^2 + 4\frac{k(k-1)}{n(n-1)}\left(\cos^2 2\theta + \sin^2 2\theta \left|\left\langle \vec{a}, \vec{b} \right\rangle\right|^2\right).
	\end{multline}
	
	To prove \eqref{eq:schordMain}, first expand the terms on the left hand side:
	\begin{multline}
		 \left(\cos^2\theta\sum_{j=1}^k a_{\sigma(j)}\bar{a}_{\sigma(j)} - \sin^2\theta\sum_{j=1}^k b_{\sigma(j)}\bar{b}_{\sigma(j)} \right)^2  \\
		= \cos^4\theta\sum_{j=1}^k a_{\sigma(j)}^2\bar{a}_{\sigma(j)}^2 \,+\,2 \cos^4\theta\sum_{1 \leq i< j \leq k} a_{\sigma(i)}\bar{a}_{\sigma(i)} a_{\sigma(j)}\bar{a}_{\sigma(j)} \\ 
		\label{eq:main1stexpand}  \quad - 2\sin^2\theta\cos^2\theta\sum_{i=1}^k \sum_{j=1}^k a_{\sigma(i)}\bar{a}_{\sigma(i)} b_{\sigma(j)}\bar{b}_{\sigma(j)} \\
		+2\sin^4\theta \sum_{1 \leq i< j \leq k} b_{\sigma(i)}\bar{b}_{\sigma(i)}b_{\sigma(j)}\bar{b}_{\sigma(j)} + \sin^4\theta\sum_{j=1}^k b_{\sigma(j)}^2\bar{b}_{\sigma(j)}^2
	\end{multline}
	and
	\begin{align}
		\nonumber \left|\sum_{j=1}^k a_{\sigma(j)} \bar{b}_{\sigma(j)}\right|^2 & = \left(\sum_{i=1}^k a_{\sigma(i)}\bar{b}_{\sigma(i)}\right) \left(\sum_{j=1}^k \bar{a}_{\sigma(j)}b_{\sigma(j)}\right) \\ 
		\label{eq:main2ndexpand} & = \sum_{i=1}^k a_{\sigma(i)}\bar{a}_{\sigma(i)} b_{\sigma(i)}\bar{b}_{\sigma(i)} + \mathop{\sum_{1 \leq i, j \leq k}}_{i \neq j} a_{\sigma(i)}\bar{a}_{\sigma(j)}b_{\sigma(i)}\bar{b}_{\sigma(j)}.
	\end{align}

	Using \eqref{eq:main1stexpand} and \eqref{eq:main2ndexpand} to re-write the left hand side of \eqref{eq:schordMain} yields
	\begin{multline}\label{eq:mainrewrite}
		 \frac{4}{n!}\sum_{\sigma \in S_n} \left[ \cos^4\theta\sum_{j=1}^k a_{\sigma(j)}^2\bar{a}_{\sigma(j)}^2 + 2 \cos^4\theta\sum_{1 \leq i< j \leq k} a_{\sigma(i)}\bar{a}_{\sigma(i)} a_{\sigma(j)}\bar{a}_{\sigma(j)} \right. \\
		- 2\sin^2\theta\cos^2\theta\sum_{i=1}^k \sum_{j = 1}^k a_{\sigma(i)}\bar{a}_{\sigma(i)} b_{\sigma(j)}\bar{b}_{\sigma(j)} + 2\sin^4\theta \sum_{1 \leq i< j \leq k} b_{\sigma(i)}\bar{b}_{\sigma(i)} b_{\sigma(j)}\bar{b}_{\sigma(j)}  \\
		+ \sin^4\theta\sum_{j=1}^k b_{\sigma(j)}^2\bar{b}_{\sigma(j)}^2   + 4\sin^2\theta\cos^2\theta \sum_{j=1}^k a_{\sigma(j)}\bar{a}_{\sigma(j)} b_{\sigma(j)}\bar{b}_{\sigma(j)} \\
		\left. +\, 4\sin^2\theta\cos^2\theta \mathop{\sum_{1 \leq i, j \leq k}}_{i \neq j} a_{\sigma(i)}\bar{a}_{\sigma(j)}b_{\sigma(i)}\bar{b}_{\sigma(j)} \right]
	\end{multline}

	Since $S_n$ is finite, we can distribute the outer sum, which we do using the following lemma.

	\begin{lemma}\label{lem:distribute}
		Let $\{c_1, \ldots , c_n\}$ be a set of $n$ elements. Then
		\begin{equation}\label{eq:dist1}
			\sum_{\sigma \in S_n} \sum_{j=1}^k c_{\sigma(j)} = k!(n-k)!{n-1 \choose k-1} \sum_{j=1}^n c_j
		\end{equation}
		and
		\begin{equation}\label{eq:dist2}
			\sum_{\sigma \in S_n} \sum_{1 \leq i< j \leq k} c_{\sigma(i)} c_{\sigma(j)} = k!(n-k)! {n-2 \choose k-2} \sum_{1 \leq i< j \leq n} c_i c_j
		\end{equation}
	\end{lemma}

	\begin{proof}
		First, notice that any rearrangement of the $k$-element set $\{c_{\sigma(1)}, \ldots , c_{\sigma(k)}\}$ or of its $(n-k)$-element complement doesn't change the sum $\sum_{j=1}^k c_{\sigma(j)}$, so each $\sum_{j=1}^k c_{\sigma(j)}$ is repeated $k!(n-k)!$ times in \eqref{eq:dist1}, and likewise for each $\sum_{i < j} c_{\sigma(i)} c_{\sigma(j)}$ in \eqref{eq:dist2}.

		Moreover, for each fixed index $m$, the term $c_m$ appears in exactly ${n-1 \choose k-1}$ sums of the form $\sum_{j=1}^k c_{\sigma(j)}$, which implies \eqref{eq:dist1}. Similarly, for each fixed $m$ and $s$, the term $c_m c_s$ appears in ${n-2 \choose k-2}$ sums of the form $\sum_{i < j} c_{\sigma(i)}c_{\sigma(j)}$, so \eqref{eq:dist2} follows.
	\end{proof}

	Using Lemma~\ref{lem:distribute} and
	\[
		\sum_{i = 1}^k \sum_{j=1}^k a_{\sigma(i)}\bar{a}_{\sigma(i)} b_{\sigma(j)}\bar{b}_{\sigma(j)} = \sum_{i=1}^k a_{\sigma(i)}\bar{a}_{\sigma(i)} b_{\sigma(i)}\bar{b}_{\sigma(i)} + \mathop{\sum_{1 \leq i, j \leq k}}_{i \neq j} a_{\sigma(i)}\bar{a}_{\sigma(i)} b_{\sigma(j)}\bar{b}_{\sigma(j)},
	\]
	we can re-write \eqref{eq:mainrewrite} as 
	\begin{multline}\label{eq:mainrewrite2}
		4\frac{k!(n-k)!}{n!}\left[{n-1 \choose k-1}\cos^4\theta \sum_{j=1}^n a_j^2\bar{a}_j^2 +2 {n-2 \choose k-2}\cos^4 \theta \sum_{1 \leq i< j \leq n} a_i\bar{a}_i a_j\bar{a}_j \right. \\
		- 2 {n-1 \choose k-1}\sin^2\theta \cos^2 \theta \sum_{j=1}^n a_j\bar{a}_jb_j\bar{b}_j -  2 {n-2 \choose k-2}\sin^2 \theta \cos^2 \theta \mathop{\sum_{1 \leq i, j \leq n}}_{i \neq j} a_i\bar{a}_ib_j\bar{b}_j \\
		+ 2{n-2 \choose k-2} \sin^4 \theta \sum_{1 \leq i< j \leq n} b_i\bar{b}_ib_j\bar{b}_j + {n-1 \choose k-1} \sin^4 \theta \sum_{j=1}^n b_j^2 \bar{b}_j^2 \\
		\left. + 4 {n-1 \choose k-1} \sin^2\theta \cos^2 \theta\sum_{j=1}^n a_j\bar{a}_jb_j\bar{b}_j + 4 {n-2 \choose k-2} \sin^2\theta \cos^2\theta \mathop{\sum_{1 \leq i, j \leq n}}_{i \neq j} a_i \bar{a}_j b_i \bar{b}_j\right]. 
	\end{multline}
	From here on all sums will be from $1$ to $n$, we will use ``$i < j$'' as shorthand for ``$1 \leq i < j \leq n$'' and ``$i \neq j$'' as shorthand for ``$1 \leq i , j \leq n, i \neq j$''.

	Notice that
	\begin{multline*}
		4{n-1 \choose k-1} \sin^2\theta \cos^2 \theta\sum_{j=1}^n a_j \bar{a}_j b_j \bar{b}_j + 4 {n-2 \choose k-2} \sin^2\theta \cos^2 \theta\sum_{i \neq j} a_i\bar{a}_j b_i \bar{b}_j \\
		= 4\sin^2\theta \cos^2 \theta \left[{n-2 \choose k-1} \sum_{j=1}^n a_j \bar{a}_j b_j \bar{b}_j + {n-2 \choose k-2} \left(\sum_{j=1}^n a_j \bar{a}_j b_j \bar{b}_j +  \sum_{i \neq j} a_i \bar{a}_j\right)\right]\\
		= 4\sin^2\theta\cos^2\theta \left[{n-2 \choose k-1} \sum_{j=1}^n a_j \bar{a}_j b_j \bar{b}_j + {n-2 \choose k-2} \left(\sum_{i=1}^na_i\bar{b}_i\right)\left(\sum_{j=1}^n \bar{a}_jb_j\right)\right] \\
		 = 4\sin^2\theta \cos^2\theta \left[{n-2 \choose k-1} \sum_{j=1}^n a_j \bar{a}_j b_j \bar{b}_j + {n-2 \choose k-2}\left| \left\langle \vec{a}, \vec{b} \right\rangle\right|^2\right].
	\end{multline*}

	Therefore, \eqref{eq:mainrewrite2} can be rewritten as
	\begin{multline}\label{eq:1}
		4\frac{k!(n-k)!}{n!}\left[{n-1 \choose k-1} \cos^4 \theta\sum_{j=1}^n a_j^2\bar{a}_j^2 +2 {n-2 \choose k-2}\cos^4 \theta \sum_{1 \leq i< j \leq n} a_i\bar{a}_i a_j\bar{a}_j \right. \\
		- 2 {n-1 \choose k-1}\sin^2\theta\cos^2 \theta \sum_{j=1}^n a_j\bar{a}_jb_j\bar{b}_j - 2 {n-2 \choose k-2}\sin^2\theta \cos^2 \theta \sum_{i \neq j} a_i\bar{a}_ib_j\bar{b}_j \\
		 + 2{n-2 \choose k-2} \sin^4 \theta \sum_{i< j} b_i\bar{b}_ib_j\bar{b}_j + {n-1 \choose k-1} \sin^4 \theta \sum_{j=1}^n b_j^2 \bar{b}_j^2 \\
		\left. + 4 {n-2 \choose k-1} \sin^2\theta \cos^2 \theta \sum_{j=1}^n a_j\bar{a}_jb_j\bar{b}_j\right] + 16\frac{k(k-1)}{n(n-1)} \sin^2\theta \cos^2 \theta \left|\left\langle \vec{a}, \vec{b} \right\rangle\right|^2 .
	\end{multline}

	Now, we use the fact that $|\vec{a}|^2 = 1 = |\vec{b}|^2$ to show that this is equal to the right hand side of \eqref{eq:schordMain}; specifically, we have
	\begin{enumerate}
		\renewcommand{\theenumi}{(\roman{enumi})}
		\item \label{item:arule}$\sum_{j=1}^n a_j\bar{a}_j = 1$
		\item \label{item:brule}$\sum_{j=1}^n b_j\bar{b}_j = 1$
		% \item \label{item:iprule}$\sum_{i=1}^n a_i b_i = 0$
	\end{enumerate}

	First, we use \ref{item:arule} in the form
	\begin{equation}\label{item:arule2}
		1 = \left(\sum_{j=1}^n a_j\bar{a}_j\right)^2 = \sum_{j=1}^n a_j^2 \bar{a}_j^2 + 2 \sum_{i< j} a_i\bar{a}_ia_j\bar{a}_j
	\end{equation}
	to eliminate $\sum_{i < j} a_i\bar{a}_i a_j\bar{a}_j$ from \eqref{eq:1}, yielding
	\begin{multline*}
		4\frac{k!(n-k)!}{n!}\left[{n-2\choose k-2}\cos^4 \theta + {n-2 \choose k-1}\cos^4 \theta \sum_{j=1}^n a_j^2\bar{a}_j^2 \right.\\
		- 2{n-1 \choose k-1} \sin^2\theta \cos^2\theta \sum_{j=1}^n a_j \bar{a}_j b_j \bar{b}_j - 2 {n-2 \choose k-2} \sin^2\theta \cos^2 \theta\sum_{i \neq j} a_i\bar{a}_i b_j\bar{b}_j \\
		+ 2 {n-2 \choose k-2} \sin^4\theta\sum_{i < j} b_i \bar{b}_i b_j \bar{b}_j + {n-1 \choose k-1} \sin^4 \theta \sum_{j=1}^n b_j^2 \bar{b}_j^2 \\
		\left. + 4 {n-2 \choose k-1} \sin^2\theta \cos^2 \theta\sum_{j=1}^n a_j\bar{a}_j b_j \bar{b}_j \right]+ 16\frac{k(k-1)}{n(n-1)} \sin^2\theta \cos^2\theta\left|\left\langle \vec{a}, \vec{b} \right\rangle\right|^2.
	\end{multline*}

	Similar manipulations involving \ref{item:brule} allow us to eliminate $\sum_{i < j} b_i\bar{b}_i b_j\bar{b}_j$ from the above and get
	\begin{multline}\label{eq:2}
		4\frac{k!(n-k)!}{n!}\left[{n-2\choose k-2}\left(1-2\sin^2\theta \cos^2\theta\right) + {n-2 \choose k-1} \cos^4 \theta \sum_{j=1}^n a_j^2\bar{a}_j^2 \right.\\
		- 2{n-1 \choose k-1}\sin^2\theta \cos^2 \theta \sum_{j=1}^n a_j \bar{a}_j b_j \bar{b}_j - 2 {n-2 \choose k-2} \sin^2\theta \cos^2 \theta \sum_{i \neq j} a_i\bar{a}_i b_j\bar{b}_j \\
		\left. + {n-2 \choose k-1} \sin^4 \theta \sum_{j=1}^n b_j^2 \bar{b}_j^2  + 4 {n-2 \choose k-1} \sin^2\theta \cos^2\theta\sum_{j=1}^n a_j\bar{a}_jb_j \bar{b}_j\right] \\
		+ 16\frac{k(k-1)}{n(n-1)} \sin^2 \theta \cos^2 \theta\left|\left\langle \vec{a}, \vec{b}\right\rangle\right|^2,
	\end{multline}
	where the constant coefficient follows from
	\[
		1 = \left(\sin^2\theta + \cos^2\theta \right)^2 = \sin^4 \theta + 2\sin^2\theta \cos^2 \theta + \cos^4 \theta.
	\]

	Next, to eliminate $\sum_{i \neq j}^n a_i\bar{a}_i b_j\bar{b}_j$, note that multiplying \ref{item:arule} and \ref{item:brule} yields
	\begin{equation}\label{eq:abrule}
		1 = \left(\sum_{i=1}^n a_i\bar{a}_i\right)\left(\sum_{j=1}^n b_j\bar{b}_j\right) = \sum_{i=1}^n \sum_{j=1}^n a_i\bar{a}_ib_j\bar{b}_j = \sum_{j=1}^n a_j\bar{a}_j b_j\bar{b}_j + \sum_{i \neq j} a_i\bar{a}_ib_j\bar{b}_j.
	\end{equation}

	Hence, \eqref{eq:2} can be re-written as
	\begin{multline}\label{eq:3}
		4\frac{k!(n-k)!}{n!}\left[ {n-2 \choose k-2} \left(1-4\sin^2\theta \cos^2 \theta\right) + {n-2 \choose k-1} \cos^4 \theta \sum_{j=1}^n a_j^2\bar{a}_j^2 \right.\\
		- 2{n-2 \choose k-1} \sin^2 \theta \cos^2 \theta \sum_{j=1}^n a_j \bar{a}_j b_j \bar{b}_j + {n-2 \choose k-1} \sin^4 \theta \sum_{j=1}^n b_j^2 \bar{b}_j^2 \\
		\left.+ 4 {n-2 \choose k-1} \sin^2\theta \cos^2\theta\sum_{j=1}^n a_j\bar{a}_jb_j \bar{b}_j\right] + 16\frac{k(k-1)}{n(n-1)}\sin^2\theta \cos^2 \theta\left|\left\langle \vec{a}, \vec{b}\right\rangle\right|^2.
	\end{multline}

	Combining the third and fifth terms, then, yields
	\begin{multline*}
		4\frac{k(k-1)}{n(n-1)}\left(1-4\sin^2 \theta \cos^2 \theta \right) + \frac{k!(n-k)!}{n!}{n-2 \choose k-1}\left[\cos^4 \theta\sum_{j=1}^n a_j^2 \bar{a}_j^2 \right. \\
		\left. + 2\sin^2 \theta \cos^2 \theta\sum_{j=1}^n a_j \bar{a}_j b_j \bar{b}_j + \sin^4 \theta \sum_{j=1}^n b_j^2 \bar{b}_j^2 \right] + 16\frac{k(k-1)}{n(n-1)} \sin^2 \theta \cos^2 \theta \left|\left\langle \vec{a}, \vec{b} \right\rangle \right|^2
	\end{multline*}
	or, equivalently,
	\begin{multline*}
		\frac{k(n-k)}{n(n-1)} \sum_{j=1}^n\left(\left|\sqrt{2} \cos \theta a_j\right|^2 + \left|\sqrt{2} \sin \theta b_j\right|^2 \right)^2 \\
		+ 4\frac{k(k-1)}{n(n-1)}\left(1 - 4\sin^2\theta \cos^2 \theta + 4 \sin^2 \theta \cos^2 \theta \left|\left\langle \vec{a}, \vec{b} \right\rangle\right|^2\right).
	\end{multline*}
	
	Using the double angle identity $\sin 2 \theta = 2 \sin \theta \cos\theta$ and the fact that the $j$th edge length $|e_j| = \left|\sqrt{2} \cos \theta a_j\right|^2 + \left|\sqrt{2} \sin \theta b_j\right|^2$, the above is just
	\[
		\frac{k(n-k)}{n(n-1)}\sum_{j=1}^n |e_j|^2 + 4 \frac{k(k-1)}{n(n-1)}\left(1 - \sin^2 2\theta + \sin^2 2\theta \left|\left\langle \vec{a}, \vec{b}\right\rangle\right|^2\right),
	\]
	which is equivalent to the right hand side of \eqref{eq:schordMain}.
\end{proof}

\clearpage

\bibliography{polygrass}

\end{document}